\documentclass[12pt]{amsart}
\usepackage{amssymb}
\usepackage{verbatim}
\usepackage[usenames]{color}
\usepackage{hyperref}
\usepackage{url}
\usepackage{accents} 
\usepackage[all]{xy}

\newtheorem*{rep@theorem}{\rep@title}
\newcommand{\newreptheorem}[2]{%
\newenvironment{rep#1}[1]{%
 \def\rep@title{#2 \ref{##1}}%
 \begin{rep@theorem}}%
 {\end{rep@theorem}}}

\newtheorem{theorem}{Theorem}[section]
\newreptheorem{theorem}{Theorem}
\newreptheorem{corollary}{Corollary} 
\newtheorem{proposition}[theorem]{Proposition} 
\newtheorem{observation}[theorem]{Observation} 
\newtheorem{lemma}[theorem]{Lemma}

\newtheorem{cor}[theorem]{Corollary}
\newtheorem{fact}[theorem]{Fact}

\theoremstyle{definition}
\newtheorem{definition}[theorem]{Definition}

\newtheorem{question}[theorem]{Question}

\theoremstyle{remark}
\newtheorem{example}[theorem]{Example} 
\newtheorem{remark}[theorem]{Remark}

\newcommand{\zfc}{\tn{ZFC}}
\newcommand{\zf}{\tn{ZF}}

\newcommand{\ad}{\tn{AD}}

\newcommand{\mc}{\mathcal}
\newcommand{\tn}{\textnormal}
\newcommand{\mbb}{\mathbb}
\newcommand{\proves}{\vdash}

\newcommand{\baire}{{{^\omega}\omega}}

\title[On $(\Sigma^2_1)^{\tn{uB}}$
 absoluteness between $\tn{V}$ and $\tn{HOD}$]
 {On $(\Sigma^2_1)^{\tn{uB}}$
 absoluteness between $\tn{V}$ and $\tn{HOD}$}

\author{Gabriel Goldberg}
\address{
Gabriel Goldberg \\
Evans Hall \\
UC Berkeley\\
Berkeley, CA 94702, U.S.A.}
\email{ggoldberg@berkeley.edu}
\author{Dan Hathaway}
\address{
Dan Hathaway \\
Mathematics Department \\
University of Vermont\\
Burlington, VT 05401, U.S.A.}
\email{Daniel.Hathaway@uvm.edu}

\begin{document}

\begin{abstract}
We put together Woodin’s $\Sigma^2_1$ basis theorem of $\ad^+$ and Vop\v{e}nka's theorem to conclude the following:
If there is a proper class of Woodin cardinals,
 then every $(\Sigma^2_1)^{\tn{uB}}$ statement that is true in $V$
 is true in $\tn{HOD}$ (Theorem~\ref{intro_main_downwards_result_no_param}).
Moreover, this is true even if we allow a parameter $C \subseteq \mbb{R}$
 such that $C$ and its complement have scales that are $\tn{OD}$
 and universally Baire
 (Theorem~\ref{downward_to_hod_without_pc_woodins}).
We also investigate whether \((\Sigma^2_1)^{\tn{uB}}\) statements are upwards absolute from \(\tn{HOD}\) to \(V\) under large cardinal hypotheses, observing that this is true if \(\tn{HOD}\) has a proper class of Woodin cardinals.
Finally, we discuss
 $(\forall^{\baire})\, (\Sigma^2_1)^{\tn{uB}}$ absoluteness and conclude
 that this much absoluteness between $\tn{HOD}$ and $V$ cannot be
 implied by any large cardinal axiom consistent with the
 axiom ``$V = \tn{Ultimate}$ $L$''
 (Corollary~\ref{lc_con_with_ultl_cannot_imply_super_abs}).
\end{abstract}

\maketitle

\section{Introduction}
Shoenfield's absoluteness theorem states that
any \(\Sigma^1_2\)-property is absolute between the constructible universe and the whole universe of sets. This paper studies the analogous question for the inner model \(\tn{HOD}\).

We begin by proving the folklore result that
assuming Projective Determinacy, HOD is projectively correct. We then seek out the strongest absoluteness result that can be proved for HOD under large cardinal hypotheses in \(V\).

A set $A\subseteq \baire$ is universally Baire if for all compact Hausdorff spaces \(X\) and continuous functions \(\pi: X\to \baire\), $\pi^{-1}[A]$ has the property of Baire. Under large cardinal hypotheses, the universally Baire sets are the sets that are absolutely definable in all generic extensions of the universe of sets. (See the Tree Production Lemma
 \cite[Theorem 4.2]{steel_dm} for a precise statement).

Our main theorem is that assuming the existence of a proper class of Woodin cardinals, 
any \((\Sigma^2_1)^{\tn{uB}}\) statement is downwards absolute to HOD. 
If HOD itself has a proper class of Woodin cardinals, then upwards absoluteness is true as well. 

Let $\Gamma^\infty$ be the pointclass
 of all universally Baire sets of reals.
Given a set of reals $B$,
 we say that the statement $\varphi$
 is $(\Sigma^2_1)^{\tn{uB}}$
 using $B$ as a parameter iff
 there is a formula $\psi$ such that
 $\varphi$ is equivalent to
 $$(\exists A \in \Gamma^\infty)
 \bigg[
 \langle H(\omega_1), \in, A, B \rangle
 \models \psi(A,B)
 \bigg].$$
We say that a statement $\varphi$
 is $(\Sigma^2_1)^{\tn{uB}}$
 iff it is $(\Sigma^2_1)^{\tn{uB}}$
 using the empty set as a parameter.

Woodin showed that assuming a proper
 class of Woodin cardinals,
 $(\Sigma^2_1)^{\tn{uB}}$ statements
 are absolute between $V$ and every
 (set) forcing extension of $V$.
Our main result is the following:
\begin{reptheorem}
{intro_main_downwards_result_no_param}
Assume there is a proper class
 of Woodin cardinals.
Let $\varphi$ be a true
 $(\Sigma^2_1)^{\tn{uB}}$ statement.
Then $\tn{HOD}$ satisfies $\varphi$.
\end{reptheorem}

The other direction is much easier,
 provided that now $\tn{HOD}$
 has a proper class of Woodin cardinals
 instead of $V$:

\begin{reptheorem}{upwards_theorem}
Assume that $\tn{HOD}$ has a proper class
 of Woodin cardinals.
Let $\varphi$ be a $(\Sigma^2_1)^{\tn{uB}}$
 statement that is true in $\tn{HOD}$.
Then $V$ satisfies $\varphi$.
\end{reptheorem}

\begin{question}
\label{intro_question_upwards_abs_from_woodins_in_v}
If $V$ has a proper class of
 Woodin cardinals and
 $\varphi$ is a
 $(\Sigma^2_1)^{\tn{uB}}$ statement
 that is true in $\tn{HOD}$,
 must $V$ satisfy $\varphi$?
\end{question}

Our main theorem is actually stronger,
 because we allow a certain kind of parameter:
\begin{reptheorem}
{downwards_with_strongly_od_param_with_woodins}
Assume there is a proper class of Woodin
 cardinals.
Let $C$ be a set of reals such that it
 and its complement have scales
 which are both $\tn{OD}$ and $\tn{uB}$.
Let $\varphi(C)$ be a true
 $(\Sigma^2_1)^{\tn{uB}}$ statement
 that uses $C$ as a parameter.
Then $\varphi(C \cap \tn{HOD})$ is true in $\tn{HOD}$.
\end{reptheorem}

The hypothesis that $C$ and its complement
 have scales which are
 $\tn{OD}$ and $\tn{uB}$ is conspicuous.
It is a central concept of this paper.
Here is the main open question we have
 about this notion:

\begin{question}
Assume there is a proper class of
 Woodin cardinals.
Let $C$ be a set of reals that is
 $\tn{uB}$ as witnessed by
 trees in $\tn{HOD}$.
Must $C$ and its complement have
 scales which are \tn{OD} and \tn{uB}?
\end{question}

By Theorem~\ref{conj_down_thm_iff_ub_closed_in_hod},
 if $\tn{HOD}$ has a proper class of
 Woodin cardinals,
 then if $C$ is a set of reals that is
 uB via $\tn{OD}$ trees,
 then $C$ and its complement have scales which
 are \tn{OD} and \tn{uB}.
Thus the question of whether $\tn{HOD}$
 has a proper class of Woodin cardinals
 is related to our open questions
 in this paper.
Let us elaborate:

\begin{definition}
The HOD hypothesis is the statement
 that there is a proper class of regular 
 cardinals that are not $\omega$-strongly measurable in $\tn{HOD}$.
See Definition 3.35 of \cite{woodin_in_search_of_ult_l}
 for the definition of $\omega$-strongly measurable in $\tn{HOD}$.
\end{definition}

The exact nature of the definition of the
 HOD hypothesis is not important.
We are only interested in the HOD hypothesis
 in the case that there is an extendible cardinal.
Here there are some beautiful equivalences,
 such as the following:
 If there is an extendible cardinal $\delta$,
 then the HOD hypothesis holds iff
 every singular cardinal $\gamma > \delta$
 is singular in $\tn{HOD}$ and $\gamma^+ = (\gamma^+)^\tn{HOD}$
 (see Theorem 3.39 of \cite{woodin_in_search_of_ult_l}).

\begin{fact}
If $\delta$ is an extendible cardinal and the
 \tn{HOD} hypothesis
 holds, then every Woodin cardinal $\kappa > \delta$ is Woodin in $\tn{HOD}$.
\end{fact}
\begin{proof}
See Theorem 3.39 of \cite{woodin_in_search_of_ult_l}
 for the fact that
 if $\delta$ is an extendible cardinal and the
 $\tn{HOD}$ hypothesis holds, then
 $\tn{HOD}$ is a weak extender model
 for the supercompactness of $\delta$.
Then, Woodin proved (unpublished)
 that any weak extender model for the
 supercompactness of some cardinal $\delta$
 satisfies the $\delta$-cover and
 $\delta$-approximation properties.
Finally, the paper
 \cite{hamkins_cover_approx}
 analyzes the $\delta$-cover and
 $\delta$-approximation properties.
That paper shows that if an inner model $M$
 satisfies these properties and
 $\kappa > \delta$ is Woodin,
 then $\kappa$ is Woodin in $M$
 (see Corollary 25 of \cite{hamkins_cover_approx}).
\end{proof}

If $V$ has an extendible cardinal $\kappa$,
 then there is a proper class of Woodin
 cardinals.
Thus if in addition
 the HOD hypothesis holds,
 then $\tn{HOD}$ has a proper class of Woodin cardinals.

Thus, if there is a universe $V$ with
 an extendible cardinal where
 Theorem~\ref{downwards_with_strongly_od_param_with_woodins}
 \textit{fails} when we replace the hypothesis
 that $C$ and its complement have scales which
 are \tn{OD} and \tn{uB}
 with the hypothesis that $C$ is $\tn{uB}$ via $\tn{OD}$ trees,
 then the HOD hypothesis must fail
 in $V$.

Similarly, if there is a universe $V$
 with an extendible cardinal where
 $(\Sigma^2_1)^{\tn{uB}}$ statements
 are \textit{not} upwards absolute from $\tn{HOD}$ to $V$,
 then the HOD Hypothesis must fail
 in $V$.

Our Theorem~\ref{downward_to_hod_without_pc_woodins}
 implies that
 \tn{HOD} is close to $V$
 in a certain sense,
 even in the case that the
 HOD Hypothesis fails!
The theorem shows that if $V$ has a proper class
 of Woodin cardinals, then
 $$\mbox{$V$ does not
 transcend $\tn{HOD}$ in terms of what
 mice exist.}$$
For example, if $V$ has 
 a countable premouse $\mathcal{M}$
 satisfying a certain theory
 $T \in \tn{HOD}$ and there is a
 universally Baire
 $\omega_1$-iteration strategy
 $\Sigma \subseteq \mbb{R}$ for $\mathcal{M}$
 in $V$,
 then there is a premouse
 $\mathcal{M}' \in \tn{HOD}$
 that is countable in $\tn{HOD}$
 such that $\mathcal{M}' \models T$
 and $\tn{HOD}$
 satisfies that there exists a
 universally Baire
 $\omega_1$-iteration strategy
 $\Sigma' \subseteq \mbb{R}$ for $\mathcal{M}'$.

A relative of a statement of the form
 ``the universe contains a certain countable premouse
 with a universally Baire iteration strategy''
 is a statement of the form
 ``$\zfc \proves_\Omega \varphi$''
 (in the sense of $\Omega$ logic).
Let us briefly review $\Omega$-logic here
 (see \cite{omega_logic_primer} for more details).
Let $T$ be a theory
 and $\varphi$ be a statement in the language
 of set theory.
We say that $\varphi$ is an $\Omega$-validity of $T$
 (or is an $\Omega$-truth of $T$)
 and write $T \models_\Omega \varphi$ iff every
 rank initial segment of every (set) forcing
 extension of $V$ which satisfies $T$ also
 satisfies $\varphi$.
Also, we say that a $\Pi_2$ statement $\psi$
 is $\Omega$-valid iff
 $\psi$ is true in every (set) forcing
 extension of $V$.
Note that given a definable theory $T$
 and a statement $\varphi$,
 the statement ``$T \models_\Omega \varphi$''
 is asserting that a certain $\Pi_2$
 statement holds in every (set) forcing
 extension of $V$.

Given a universally Baire set of reals $A \subseteq \mbb{R}$
 and a countable transitive model $M$ of $\zfc$,
 we say that $M$ is strongly $A$-closed iff
 for every $\mbb{P} \in M$ and $G$ that is
 $\mbb{P}$-generic over $M$, we have
 $A \cap M[G] \in M[G]$.
We say that
 $\varphi$ is an $\Omega$-theorem of $T$
 (or $\varphi$ is $\Omega$-provable from $T$)
 and write $T \proves_\Omega \varphi$
 iff there is a universally Baire set of reals $A$
 such that every countable transitive model $M$ of $T$
 that is strongly $A$-closed is such that
 $M \models (T \models_\Omega \varphi)$.
The notion of $\Omega$-provability is meant to
 capture the notion of a statement being true in all
 models that are closed under certain mouse operators.

$\Omega$-logic can also handle formulas with real parameters
 as follows: If $T$ is a theory, $x$ is a real number,
 and $\varphi$ is a formula, then
 $T \proves_\Omega \varphi(x)$ means
 there exists a universally Baire set of reals $A$
 such that for every
 countable transitive model $M$ of $T$
 such that $x \in M$ and
 $M$ is strongly $A$-closed,
 we have
 $M \models (T \models_\Omega \varphi(x))$.

It is true that
 if there is a proper class of inaccessible cardinals and
 $T \proves_\Omega \varphi$,
 then $T \models_\Omega \varphi$;
 this is the \textit{soundness of $\Omega$-logic}.
The $\Omega$ conjecture says that the converse holds,
 assuming there is a proper class of Woodin cardinals.
More precisely, the $\Omega$ conjecture says that
 if there is a proper class of Woodin cardinals, then
 $\zfc \models_\Omega \varphi$ iff
 $\zfc \proves_\Omega \varphi$
 for any statement $\varphi$ in the language of set theory.
The $\Omega$ conjecture is important because it attempts
 to explain why certain $\Pi_2$ statements
 are true in every (set) forcing extension of $V$.
One could also say that the $\Omega$ conjecture
 attempts to show that the statements
 that are provable from large cardinals are exactly
 the statements that follow from mouse existence
 principles.
The $\Omega$-Conjecture is a specific way of saying
 that the universally Baire sets of reals
 capture essential information about $V$.

Theorem~\ref{downward_to_hod_without_pc_woodins}
 shows that if a statement
 is $\Omega$-provable in $V$,
 then it is $\Omega$-provable
 in $\tn{HOD}$
 (provided that $V$ has a proper class
 of Woodin cardinals).
This immediately tells us something
 about the $\Omega$ conjecture:
\begin{repcorollary}
{omega_conj_goes_down}
Assume there is a proper class of Woodin
 cardinals.
If the $\Omega$ conjecture is true in $V$,
 then it is true in $\tn{HOD}$.
\end{repcorollary}

The downwards absoluteness of
 $\Omega$-provability also relates
 to the concept of $V$ being $\varphi$-closed,
 in the sense of \cite{woodin_ch2}
 (when $V$ satisfies the
 $\Omega$ conjecture).
For example, we get the following:
\begin{repcorollary}
{omega_huge_go_down}
Assume there is a proper class of huge cardinals
 and the $\Omega$ conjecture holds.
Then $\tn{HOD}$ satisfies that every set belongs to
 an inner model with a huge cardinal.
\end{repcorollary}

Going beyond
 $(\Sigma^2_1)^{\tn{uB}}$ statements
 we have
 $$(\forall^\baire)(\Sigma^2_1)^{\tn{uB}}$$
 statements.
A statement $\varphi$ is of this form
 iff there is a $(\Sigma^2_1)^{\tn{uB}}$
 formula $\psi$ such that
 $$\varphi :\Leftrightarrow
 (\forall x \in \baire)\, \psi(x).$$
Let us define
 $$\mbox{generic $(\forall^\baire)(\Sigma^2_1)^{\tn{uB}}$ absoluteness}$$
 to mean that
 $(\forall^\baire)(\Sigma^2_1)^{\tn{uB}}$
 statements are absolute between $V$ and its
 (set) forcing extensions.

Because the
 $(\Sigma^2_1)^{\tn{uB}}$ generic absoluteness theorem relativizes to a real parameter,
 it follows that
 if there is a proper class of Woodin
 cardinals and some (set) forcing
 extension of $V$ satisfies a given
 $(\forall^\baire)(\Sigma^2_1)^{\tn{uB}}$
 statement $\varphi$, then $V$
 satisfies $\varphi$.

On the other hand the upwards direction
 need not hold.
Here is a typical example of this:
Let $\Phi_1$ be the statement
 that for every real $x$,
 there exists a countable premouse
 $\mathcal{M}$ containing $x$ such that
 there exists a universally Baire
 $\omega_1$-iteration strategy for $\mathcal{M}$.
Now $\Phi_1$ implies
 $\mbb{R} \subseteq \tn{HOD}$
 (see Lemma 1.21 and Corollary 1.12 \cite{sargsyan_dimt}).

We claim that $\Phi_1$ implies that
 generic $(\forall^\baire)(\Sigma^2_1)^{\tn{uB}}$
 absoluteness fails.
Suppose that $V$ satisfies $\Phi_1$
 but $V[G]$ is a set forcing extension
 of $V$ by some $\tn{OD}$
 homogeneous forcing which
 adds a real number
 (such as Cohen forcing).
Since the forcing is OD and homogeneous,
 $\tn{HOD}^{V[G]} \subseteq \tn{HOD} \subseteq V$,
 so every real in $\tn{HOD}^{V[G]}$ is in $V$.
It must be that $V[G]$ does not
 satisfy $\Phi_1$, otherwise
 if $x$ was a real in $V[G]$ but not in
 $V$ then $x$ would be in
 $\tn{HOD}^{V[G]}$ but not in $V$,
 which is impossible.

We would like a weaker statement
 that also implies that
 $(\forall^\baire)(\Sigma^2_1)^{\tn{uB}}$ absoluteness fails.
Here it is:
 Let $\Phi_2$ be the statement that
 for every real number $x$,
 $$\zfc + \mbox{``there exists a
 Woodin cardinal''} \proves_\Omega
 (x \in \tn{HOD}).$$
It is true that
 $\Phi_2$ is $(\forall^\baire)(\Sigma^2_1)^{\tn{uB}}$
 and it also implies that
 generic
 $(\forall^\baire)(\Sigma^2_1)^{\tn{uB}}$ absoluteness fails
 for basically the same reason that
 $\Phi_1$ does
 (see Proposition~\ref{phi2_implies_super_abs_fail}).
Let us record that last fact in a lemma:
\begin{lemma}
\label{cohen_forcing_breaks_phi2}
Assume there is a Woodin cardinal $\delta$
 and a proper class of inaccessible cardinals.
Let $G$ be Cohen generic over $V$.
Then $V[G] \models \neg \Phi_2$.
\end{lemma}
\begin{proof}
Since the forcing is OD and homogeneous,
 $\tn{HOD}^{V[G]} \subseteq V$.
Now let $x \in \mathbb{R}^{V[G]}$
 be the Cohen real that is added.
It must be that in $V[G]$,
 $\zfc$ + ``there exists a Woodin cardinal''
 $\not\proves_\Omega (x \in \tn{HOD})$.
If this was not the case, then
 using the soundness of $\Omega$-logic
 (because there is a proper class
 of inaccessible cardinals)
 in $V[G]$,
 we would have
 $\zfc$ + ``there exists a Woodin cardinal''
 $\models_\Omega (x \in \tn{HOD})$.
Now fix an inaccessible cardinal $\lambda > \delta$.
We have that $(V_\lambda)^{V[G]}$ satisfies the theory
 $\zfc$ + ``there exists a Woodin cardinal.''
Thus $(V_\lambda)^{V[G]} \models (x \in \tn{HOD})$.
This implies $V[G] \models (x \in \tn{HOD})$.
That is, $x \in \tn{HOD}^{V[G]}$,
 which is a contradiction because $x \not\in V$.
\end{proof}

Note that $\Phi_1$ implies $\Phi_2$
 (see Proposition~\ref{reals_in_mice_are_omega_reals}).
However, what makes $\Phi_2$ so useful
 is that the axiom ``$V$ = Ultimate $L$''
 implies $\Phi_2$
 (see Theorem~\ref{ult_l_implies_exery_real_is_omega}).
That is due to Woodin,
 but for completeness we provide
 a proof.
Here is what we mean by this axiom
 (see \cite{woodin_in_search_of_ult_l} for more
 about this axiom):
\begin{definition}
The axiom ``$V$ = Ultimate $L$''
 is the statement that there exists a proper class
 of Woodin cardinals and for every true
 $\Sigma_2$ statement $\varphi$,
 there exists a universally Baire set of reals
 $A \subseteq \mathbb{R}$ such that
 $\tn{HOD}^{L(A,\mathbb{R})}$ satisfies $\varphi$.
\end{definition}

This tells us information about the
 absoluteness of
 $(\forall^\baire)(\Sigma^2_1)^{\tn{uB}}$
 statements between
 $\tn{HOD}$ and $V$.
\begin{cor}
\label{lc_con_with_ultl_cannot_imply_super_abs}
Let $\psi$ be a statement that
 is consistent with the axiom ``$V$ = \tn{Ultimate} $L$''
 but it is also preserved by adding a Cohen real.
Then $\psi$ holding in both $V$ and $\tn{HOD}$
 cannot imply 
 $(\forall^\baire)(\Sigma^2_1)^{\tn{uB}}$ absoluteness
 between $V$ and $\tn{HOD}$.
\end{cor}
\begin{proof}
Suppose that $V$ satisfies the
 axiom ``$V$ = Ultimate $L$'' together with $\psi$.
Now $V$ satisfies $\Phi_2$
 (see Theorem~\ref{ult_l_implies_exery_real_is_omega}).
Let $G$ be generic over $V$ for Cohen forcing.
Thus $V[G]$ satisfies $\psi$.
Now $\tn{HOD}^{V[G]} \subseteq V$.
In fact, the axiom ``$V$ = Ultimate $L$'' implies that
 $\tn{HOD}^{V[G]} = \tn{HOD} = V$
 (see Theorem 7.27 of \cite{woodin_in_search_of_ult_l}).
By Lemma~\ref{cohen_forcing_breaks_phi2},
 $V[G]$ does not satisfy $\Phi_2$.
So now $V$ and $V[G]$ both satisfy $\psi$
 but $V[G]$ satisfies the statement
 $$\mbox{``$\Phi_2$ holds in $\tn{HOD}$ but fails in \tn{V}.''}$$
We are done because $\Phi_2$ is a
 $(\forall^\baire)(\Sigma^2_1)^{\tn{uB}}$ statement.
\end{proof}

So by the corollary above
 no large cardinal axiom consistent with the
 axiom ``$V$ = \tn{Ultimate} $L$'',
 holding in both $\tn{V}$ and $\tn{HOD}$, can imply
 $(\forall^\baire)(\Sigma^2_1)^{\tn{uB}}$ absoluteness
 between $V$ and $\tn{HOD}$.

\begin{remark}
We could make the following definition:
 An inner model $H$ is a \textit{weak ground} iff
 every set is (set) generic over $H$.
The inner model $\tn{HOD}$ is a weak ground by
 Theorem~\ref{theorem_vopenka}.
The upwards absoluteness result we discuss in this paper
 (Theorem~\ref{upwards_theorem}) is true if we replace
 $\tn{HOD}$ with any weak ground.
On the other hand,
 our proof of the downwards absoluteness result
 (Theorem~\ref{downward_to_hod_without_pc_woodins}, our main result)
 uses more than just that the target inner model $H$ is a weak ground.
For example, the proof also uses that $\tn{HOD} \subseteq H$
 in an essential way.
\end{remark}

\section{Preliminaries}

See \cite{feng_ub} for information
 about universally Baire sets of reals.
We will review some of that here.

\begin{definition}
Let $\lambda$ be an ordinal.
A set of reals $A \subseteq \baire$
 is ${<}\lambda$-\textit{universally Baire}
 (or ${<}\lambda$-uB) iff there is
 an ordinal $\eta$ and trees
 $T,S$ on $\omega \times \eta$
 such that $p[T] = A$
 and whenever $G$ is generic over $V$
 by a poset of size $<\lambda$,
 then $$V[G] \models p[T] = \baire\, \backslash\, p[S].$$
\end{definition}

In the definition above, we say that
 $T$ and $S$ project to complements in $V[G]$.
The definition of a set of reals being
 ${\le}\lambda$-universally Baire is similar but
 it holds for all $G$ that are generic
 over $V$ by posets of size $\le \lambda$.

\begin{remark}
\label{remark_equiv_less_lambda_ub}
An equivalent definition $A$ being
 ${<}\lambda$-universally Baire
 is as follows: for every $\gamma < \lambda$
 there is an ordinal $\eta$ and
 are trees $T,S$ on $\omega \times \eta$
 such that $p[T] = A$ and whenever $G$ is generic over $V$
 by a poset of size ${\le}\gamma$, then
 $T$ and $S$ project to complements in $V[G]$.
This equivalence follows by using the Axiom of Choice
 to take the ``disjoint union'' of trees.
See Theorem 4.2 of \cite{steel_dm}
 for an example of this argument.
\end{remark}

\begin{definition}
A set of reals is universally Baire (or \tn{uB})
 iff it is ${<}\lambda$-universally Baire
 for every ordinal $\lambda$.
We write $\Gamma^\infty$ for the collection
 of all universally Baire sets of reals.
\end{definition}

We will abuse notation and talk about subsets $B$
 of $\mbb{R}$ being universally Baire,
 by which we mean $B$ is coded in some standard way
 by some $A \subseteq \baire$ that is universally Baire.

To check if a set of reals is ${\le}\lambda$-universally Baire,
 it might seen that we need to quantify over all posets.
However the homogeneity of the collapsing poset simplifies
 matters:
\begin{theorem}
Let $\lambda$ be an infinite cardinal.
A set of reals $A \subseteq \baire$ is
 ${\le}\lambda$-universally Baire iff
 there is an ordinal $\eta$ and trees $S,T$ on
 $\omega \times \eta$ such that $p[T] = A$ and
 $\tn{Col}(\omega, \lambda)$ forces that
 $T$ and $S$ project to complements.
\end{theorem}
\begin{proof}
See Theorem 2.1 of \cite{feng_ub}.
\end{proof}

So to determine if $A \subseteq \baire$
 is ${\le}\lambda$-universally Baire,
 we only need to find \textit{some}
 $G$ that is $\tn{Col}(\omega,\lambda)$-generic over $V$
 that is appropriate.
Here is another property of the $\tn{Col}(\omega,\lambda)$
 poset we will exploit:

\begin{fact}
Let $\lambda$ be an ordinal.
Let $M$ be a transitive model of $\zf$
 such that $\lambda \in M$.
Then $\tn{Col}(\omega,\lambda)^M = \tn{Col}(\omega,\lambda)$.
Let $G$ be $\tn{Col}(\omega,\lambda)$-generic over $V$.
Then $G$ is $\tn{Col}(\omega,\lambda)$-generic over $M$.

Also, if $\eta$ is an ordinal and $T,S$ are trees
 on $\omega \times \eta$ in $M$ that project to
 complements in $V[G]$, then they project to
 complements in $M[G]$.
\end{fact}

So universally Baire sets of reals can be
 reinterpreted in appropriate forcing extensions.
Here is how to reinterpret \textit{any} set of reals,
 but there is no guarantee that this is useful
 if the set of reals is not universally Baire:

\begin{definition}
If \(M\subseteq N\) are models of ZF
and \(A\in M\) is a subset of $\baire$, then the
\textit{canonical extension
 (or reinterpretation) of \(A\) to \(N\)}
is the set \[A^{N}_M = \bigcup \{(p[T])^N : T\in M\text{ and }(p[T])^M = A\}\]
\end{definition}

Changing gears,
 here are some properties of $\tn{HOD}$
 we will use.

\begin{theorem}[Vop\v{e}nka]
\label{theorem_vopenka}
Let $X$ be a set.
Then $X$ is (set) generic over $\tn{HOD}$.
\end{theorem}
\begin{proof}
See the proof of Theorem 15.46 of
 Jech \cite{jech_book}.
\end{proof}

\begin{remark}
\label{vopenka_poset_size_bound}
By inspecting the proof of Vop\v{e}nka's Theorem
 we see that for every real $x$,
 there exists a poset $\mbb{P} \in \tn{HOD}$
 such that $x$ is in a $\mbb{P}$-generic extension
 of $\tn{HOD}$,
 and $|\mbb{P}| \le 2^{2^\omega}$
 (as calculated in $V$).
 \end{remark}

Recall that given models $M$ and $N$ of $\zf$,
 we say that $M$ is a \textit{ground} of $N$ iff
 $N$ is a (set) forcing extension of $M$.

\begin{fact}[Folklore]
\label{when_v_grows_homog_hod_shrinks}
Let $V[G]$ be a (set) forcing extension of $V$
 by a homogeneous $\tn{OD}$ poset.
Then $\tn{HOD}^{V[G]}$ is a ground of $\tn{HOD}$.
\end{fact}
\begin{proof}
See Corollary 7 of \cite{davis_hod_dichotomy}.
\end{proof}

From the fact above and Theorem~\ref{theorem_vopenka},
 one can show the following.
In this next theorem
 we are not claiming that the hypothesis that the forcing
 is OD and homogeneous is necessary, but it does make the
 proof easier.
\begin{theorem}
\label{extended_vopenka}
Suppose $V[G]$ is generic over $V$ by an \tn{OD}
 homogeneous forcing.
Then every set $X \in V[G]$
 is (set) generic over $\tn{HOD}^V$.
\end{theorem}
\begin{proof}
By Fact~\ref{when_v_grows_homog_hod_shrinks},
 $\tn{HOD}^{V[G]}$ is a ground of $\tn{HOD}$.
So, fix a set of ordinals $A \in \tn{HOD}$ such that
 $\tn{HOD}^{V[G]}[A] = \tn{HOD}$.
Let $X' \in V[G]$ be a set of ordinals such that
 $X \in L[X']$.
The pair $(A,X')$ is generic over $\tn{HOD}^{V[G]}$
 by applying Vop\v{e}nka's Theorem (Theorem~\ref{theorem_vopenka})
 in $V[G]$.
Now the model $\tn{HOD}^{V[G]}[A,X']$ satisfies the
 Axiom of Choice and $\tn{HOD}$ is also an inner model
 that satisfies the Axiom of Choice where
 $$\tn{HOD}^{V[G]} \subseteq \tn{HOD} \subseteq
 \tn{HOD}^{V[G]}[A,X'].$$
Thus by the intermediate model theorem
 (see Lemma 15.43 of Jech \cite{jech_book})
 it follows that $\tn{HOD}^{V[G]}[A,X']$
 is a forcing extension of $\tn{HOD}$.
In particular $X'$ is generic over $\tn{HOD}$,
 and so $X$ is generic over $\tn{HOD}$.
\end{proof}

An important concept for this paper
 is sets of reals being universally Baire,
 but ``witnessed by trees in $\tn{HOD}$.''
Here is the precise definition:

\begin{definition}
Let $\lambda$ be an ordinal.
A set of reals \(A\) is \({<}\lambda\)-uB 
 \textit{via \tn{OD} trees} iff
 there are ordinal definable trees \(T\) and \(S\) on
 \(\omega\times \eta\) for some ordinal \(\eta\) 
 such that $A = p[T]$ and
 whenever $G$ is generic over $V$ by a poset
 of size $< \lambda$, then $T$ and $S$
 project to complements in $V[G]$.
In other words for every $\gamma < \lambda$,
 the trees \(T\) and \(S\) project to complements in \(V^{\tn{Col}(\omega,\gamma)}\).

We say a set of reals $A$ is \tn{uB} \textit{via} $\tn{OD}$ \textit{trees}
 iff $A$ is ${<}\lambda$-\tn{uB} via \tn{OD} trees for every
 ordinal $\lambda$.
\end{definition}

\begin{remark}
Just as we commented in
 Remark~\ref{remark_equiv_less_lambda_ub},
 there is an equivalence of a set $B$
 being ${<}\lambda$-uB via OD trees:
 for every $\gamma < \lambda$
 there is an ordinal $\eta$ and
 ordinal definable trees $T$ and $S$
 such that $A = p[T]$
 on $\omega \times \eta$ such that whenever $G$
 is generic over $V$ by a poset of size $\le \gamma$,
 then $T$ and $S$ project to complements in $V[G]$.
Here it is important that we can use the
 Axiom of Choice in $\tn{HOD}$
\end{remark}

If a set of reals in $\tn{HOD}$
 is ${<}\lambda$-universally Baire in $\tn{HOD}$,
 then that set of reals can be reinterpreted
 in $V$ (using Vop\v{e}nka's theorem)
 and it is ${<}\lambda$-universally Baire in $V$:

\begin{lemma}
\label{ub_in_hod_extends_to_ub_in_v}
    If \(\bar{A}\in \tn{HOD}\) is a \({<}\lambda\)-\tn{uB} set of reals in \(\tn{HOD}\)
    where \(\lambda\) is a strong limit cardinal,
    then \(\bar{A}^{V}_\tn{HOD}\)
    is \({<}\lambda\)-\tn{uB} via $\tn{OD}$ trees.
    \begin{proof}
        Fix \(\gamma < \lambda\)
        and let \(T\) and \(S\) witness that
        \(\bar{A}\) is \({\le} \rho\)-uB in \(\tn{HOD}\)
        where \(\rho = |\mc{P}^{\text{OD}}(\mc{P}(\gamma))|\).
        We claim that \(T\) and \(S\) witness
        that \(A := \bar{A}^{V}_\tn{HOD}\) is 
        \(\gamma\)-uB in \(V\).

        Let us first verify that
        \(T\) and \(S\) project to complements in
        \(V^{\tn{Col}(\omega,\gamma)}\).
        Let $G$ be \(\tn{Col}(\omega,\gamma)\)-generic over $V$.
        By Theorem~\ref{extended_vopenka},
        every $x \in \mathbb{R}^{V[G]}$ is generic over $\tn{HOD}$
        (not just $\tn{HOD}^{V[G]}$).
        Indeed, the cardinality of the poset needed is at most $\rho$,
        by the same idea as in Remark~\ref{vopenka_poset_size_bound}.
        Since \(T\) and \(S\) witness that \(\bar{A}\) is
        \({\le}\rho\)-uB in \(\tn{HOD}\), it follows that
        \(x\in (p[T]\cup p[S])^{V[G]}\). Thus
        \(T\) and \(S\) project to complements in 
        \(V^{\tn{Col}(\omega,\gamma)}\), as desired.

        Finally, we must check that
        \(p[T]= A\). By definition,
        \(p[T]\subseteq A\).
        Suppose \(x\in A\setminus p[T]\).
        Then \(x\in p[S]\) since \(T\) and
        \(S\) project to complements appropriately. 
        Since \(x\in A\), there is some \(T'\in \tn{HOD}\)
        such that
        \((p[T'])^{\tn{HOD}} = \bar{A}\) and \(x\in p[T']\).
        Since \(p[S]\cap p[T']\neq \emptyset\),
        by absoluteness \((p[S]\cap p[T'])^\tn{HOD}\neq \emptyset\), which contradicts that \(T'\) and \(S\) project to complements in \(\tn{HOD}\).
    \end{proof}
\end{lemma}

\begin{lemma}
Fix $A \subseteq \baire$.
The following are equivalent:
\item[1)] $A$ is $\tn{uB}$ via $\tn{OD}$ trees.
\item[2)] There is a set of reals $\bar{A}$
 in $\tn{HOD}$ that is $\tn{uB}$ in $\tn{HOD}$
 such that $A$ is the reinterpretation
 of $\bar{A}$ in $V$.
\end{lemma}
\begin{proof}
The 2) $\Rightarrow$ 1) direction is the harder one,
 and it follows from Lemma~\ref{ub_in_hod_extends_to_ub_in_v}.

For the 1) $\Rightarrow$ 2) direction,
 if $T,S$ are trees in $\tn{HOD}$ that project to complements
 in a $\tn{Col}(\omega,\lambda)$
 extension $V[G]$ of $V$, then they also project to complements
 in the $\tn{Col}(\omega,\lambda)$ extension
 $\tn{HOD}[G]$ of $\tn{HOD}$.
\end{proof}

The following is a trivial but important observation:
\begin{fact}
Fix $A \subseteq \baire$
 and suppose it is $\tn{uB}$ via $\tn{OD}$ trees.
Suppose $\bar{A}$ is a set of reals in $\tn{HOD}$
 which is $\tn{uB}$ in $\tn{HOD}$ such that
 $A$ is the reinterpretation of $\bar{A}$ in $V$.
Then
 $$\bar{A} = A \cap \tn{HOD}.$$
\end{fact}

Next, Woodin proved that if there is a proper class of
 Woodin cardinals, then if $G$ is generic over $V$,
 then there exists an elementary embedding
 $j : L(\mbb{R}) \to L(\mbb{R})^{V[G]}$.
We need a generalization of this result
 (Theorem~\ref{sem_engine}),
 also due to Woodin.

\begin{definition}
\label{def_enriched}
Consider the collection $\Gamma^\infty$
 of all universally Baire sets of reals.
We say the following:
\begin{itemize}
\item[1)] $\Gamma^\infty$ is \textit{closed under scales}
 iff every set in $\Gamma^\infty$ has a scale
 each norm of which is in $\Gamma^\infty$.
\item[2)] $\Gamma^\infty$ is \textit{closed under sharps} iff for each $A \in \Gamma^\infty$,
 $(A,\mathbb{R})^\# \in \Gamma^\infty$.
\item[3)] $\Gamma^\infty$ is \textit{closed under projections} iff 
 every set of reals that is projective
 in some $A \in \Gamma^\infty$ is in $\Gamma^\infty$.
\item[4)] $\Gamma^\infty$ is \textit{strongly determined}
 iff for every $A \in \Gamma^\infty$,
 $L(A,\mathbb{R}) \models \tn{AD}^+$.
\end{itemize}
\end{definition}

\begin{fact}[Martin, Steel, Woodin]
\label{fact_martin_steel_woodin}
If there is a proper class of Woodin cardinals,
 then $\Gamma^\infty$ is strongly determined
 and closed under scales, sharps, and projections.
\end{fact}
\begin{proof}
The fact that $\Gamma^\infty$ is closed under scales
 is due to Steel and can be found in \cite{steel_dm}.
The fact that the $\infty$-homogeneously Suslin sets
 of reals are closed under projections
 follows from the Martin-Steel theorem, which can be found in
 \cite{proof_of_proj_det}.
The fact that $\Gamma^\infty$ equals the set
 of $\infty$-homogeneously Suslin sets of reals
 can be found in \cite{steel_dm}.
This gives us that $\Gamma^\infty$ is closed
 under projections.
The fact that $\Gamma^\infty$ is closed under sharps
 is due to Woodin.
An important ingredient for this is the Tree Production Lemma
 (see Section 3.3 of \cite{larson_tower_book},
 including Theorem 3.3.9 for a proof that
 $\mathbb{R}^\#$ exists
 and is universally Baire).

Finally, the fact that $\Gamma^\infty$ is strongly
 determined is due to Woodin.
To show this, note that we are allowed to use that
 $\Gamma^\infty$ is closed under sharps.
Then, we can combine this with
 the fact that \textit{derived models} satisfy $\ad^+$
 and have a certain amount of reflection,
 which can be found in \cite{steel_dm}.
Here is an alternate argument:
 Fix a universally Baire set $A \subseteq \mathbb{R}$.
Then $(A,\mathbb{R})^\#$ exists and is universally Baire
 (by hypothesis).
Every set of reals Wadge-below $(A,\mathbb{R})^\#$ is universally Baire,
 and hence every set of reals in $L(A,\mathbb{R})$ is
 universally Baire.
By a result of Neeman (see \cite{neeman_handbook_chapter}),
 since there exists a Woodin cardinal,
 every universally Baire set of reals is determined,
 so $L(A,\mathbb{R}) \models \ad$.
\end{proof}

The following is a fundamental tool of our paper:

\begin{theorem}[Woodin]
\label{sem_engine}
Assume that $\Gamma^\infty$ is
 closed under scales, sharps, and projections.
Suppose that $G$ is (set) generic over $V$
 and let $A \in \Gamma^\infty$.
Let $A^*$ be the canonical reinterpretation
 of $A$ in $V[G]$.
Then there exists an elementary embedding
 $$j : L(A,\mbb{R}) \to L(A^*,\mbb{R})^{V[G]}$$
 such that $j(A) = A^*$.
\end{theorem}
\begin{proof}
This is Theorem 17 of \cite{woodin_sem}.
\end{proof}

\section{Upwards absoluteness}

Upwards absoluteness of $(\Sigma^2_1)^{\tn{uB}}$
 statements from $\tn{HOD}$ to $V$ follows immediately
 from Woodin's Theorem~\ref{sem_engine}:

\begin{proposition}
\label{upwards_theorem_enriched}
Assume that $\tn{HOD} \models$ ``$\Gamma^\infty$
 is closed under scales, sharps, and projections.''
Let $\bar{C}$ be a set of reals in $\tn{HOD}$
 that is universally Baire in $\tn{HOD}$.
Let $C$ be the canonical reinterpretation
 of $\bar{C}$ in $V$
 (so $\bar{C} = C \cap \tn{HOD}$).
Then there is an elementary embedding
 $$j : L(\bar{C}, \mathbb{R})^\tn{HOD} \to L(C, \mathbb{R})$$
 such that $j(\bar{C}) = C$.
Thus,
 $(\Sigma^2_1)^{\tn{uB}}$ statements are
 upwards absolute from $\tn{HOD}$ to $V$.
\end{proposition}
\begin{proof}
Since $V$ satisfies the Axiom of Choice,
 let $X$ be a set of ordinals such that
 $C,\mbb{R} \in L[X]$.
By Vop\v{e}nka's Theorem~\ref{theorem_vopenka},
 $X$ is generic over $\tn{HOD}$.
Since $\tn{HOD}[X]$ and $V$ have the same reals
 and $C$ is the canonical reinterpretation
 of $\bar{C}$ in $V$,
 it follows that $C$ is the canonical reinterpretation
 of $\bar{C}$ in $\tn{HOD}[X]$.
Since $X$ is generic over $\tn{HOD}$,
 by Theorem~\ref{sem_engine}
 there is an elementary embedding
 $$j : L(\bar{C}, \mbb{R})^\tn{HOD}
 \to L(C,\mbb{R})^{\tn{HOD}[X]}.$$
However note that
 $$L(C,\mbb{R})^{\tn{HOD}[X]} =
 L(C,\mbb{R})^{L[X]} =
 L(C,\mbb{R}).$$
This proves the first part of the theorem.

For the second part,
 note that if $\varphi$ is a $(\Sigma^2_1)^{\tn{uB}}$
 statement that is true as witnessed by a
 set $\bar{D}$ that is universally Baire in $\tn{HOD}$,
 then the reinterpretation $D$ of $\bar{D}$ in $V$
 witnesses that $\varphi$ is true in $V$
 (using the first part of the theorem).
Note by Lemma~\ref{ub_in_hod_extends_to_ub_in_v}
 that $D$ is universally Baire (in $V$).
\end{proof}

From this we get the following:
\begin{theorem}
\label{upwards_theorem}
Assume that $\tn{HOD}$ has a proper class
 of Woodin cardinals.
Let $\varphi$ be a $(\Sigma^2_1)^{\tn{uB}}$
 statement that is true in $\tn{HOD}$.
Then $V$ satisfies $\varphi$.
\end{theorem}

\begin{example}
If $\tn{HOD}$ has a proper class of Woodin cardinals,
 then $\tn{HOD}$ and $V$ agree about all projective
 statements.
This is because for any projective statement $\varphi$
 we have that $\neg \varphi$ is also projective, thus
 $\varphi$ and $\neg \varphi$ are both upwards absolute
 from $\tn{HOD}$ to $V$.

The theorem above is much stronger than we need for this.
That is, we do not need to use
 Lemma~\ref{ub_in_hod_extends_to_ub_in_v} to prove this.
\end{example}

\section{Common symmetric extensions of V and HOD}

A lot of the material in this section is standard,
 but we include it for completeness.
A question related to our investigation is
 how closed are the universally Baire
 sets of reals in HOD assuming there
 are large cardinals in $V$.
That is, suppose $V$ has a proper class
 of Woodin cardinals.
Are the universally Baire sets of reals
 of HOD closed under sharps?
Are they closed under scales?
Related to this, in this section
 we will show that if $V$ has a proper class
 of Woodin cardinals, then HOD has forcing extensions that
 satisfy fragments of the Axiom of Determinacy.

Let $M$ be a transitive model of $\zfc$
 and let $\delta$ be an ordinal.
Let $\mbb{R}^*$ be a collection of reals.
We call $M(\mbb{R}^*)$ a \textit{symmetric extension}
 \textit{of} $M$ \textit{by} $\tn{Col}(\omega,{<}\delta)$ iff
 $$\mbb{R}^* = \bigcup \{ \mathbb{R}^{M[G \restriction \alpha]} : \alpha < \delta \}$$
 and
 $$\mbb{R}^* = \mbb{R}^{M(\mbb{R}^*)}$$
 for some $G$ that is $\mbox{Col}(\omega,{<}\delta)$-generic
 over $M$.
Here, by $M(\mbb{R}^*)$ we mean the smallest
 transitive model $N$
 of $\zf$ such that $M \subseteq N$ and $\mathbb{R}^* \in N$.

\begin{lemma}
If $\delta$ is a strong limit cardinal, $G$
 is $\tn{Col}(\omega,{<}\delta)$-generic over \(V\),
 and $\mbb{R}^* := \bigcup_{\alpha < \lambda} \mbb{R}^{V[G\restriction \alpha]}$, then $V(\mbb{R}^*)$ is a symmetric extension of $V$.
\end{lemma}
\begin{proof}
It suffices to show that \(\mathbb R^* = \mathbb R^{V(\mathbb R^*)}\)
(note that this is immediate if \(\delta\) is strongly inaccessible since
then \(\mathbb R^* = \mathbb R^{V[G]}\)).

Recall that $H := \tn{HOD}_{V \cup \mathbb{R}^*}^{V[G]}$ is the class of all
 sets in $V[G]$ that are hereditarily definable in $V[G]$
 using only elements of $V$ and $\mathbb{R}^*$ as parameters.
Note that $H$ is an inner model
 of $V[G]$ such that $V \subseteq H$ and $\mbb{R}^* \subseteq H$
 (but it is not yet clear if $\mbb{R}^* \in H$).

We will show that $\mbb{R}^H = \mbb{R}^*$.
Certainly $\mbb{R}^* \subseteq \mbb{R}^H$.
For the other direction, fix $x \in \mbb{R}^H$.
We will show $x \in \mbb{R}^*$.
Since $x \in \mbb{R}^H$,
 by defintion of $H$ fix some $a \in V$ and $b \in \mbb{R}^*$
 such that $x$ is definable in $V[G]$
 using only $a$ and $b$ as parameters.
Fix an ordinal $\alpha < \delta$ such that
 $b \in V[G \restriction \alpha]$.
Because $V[G]$ is generic over $V[G\restriction \alpha]$
 by a homogeneous forcing and
 $a,b \in V[G\restriction \alpha]$, it follows that
 $x \in V[G\restriction \alpha]$.
Hence $x \in \mbb{R}^*$.
We have now shown $\mbb{R}^H = \mbb{R}^*$.

It is now clear that $\mbb{R}^* \in H$ (because $\mbb{R}^H \in H$).
This implies $V(\mbb{R}^*) \subseteq H$,
 because $V(\mbb{R}^*)$ is the smallest transitive model $N$
 of $\zf$ such that $V \subseteq N$ and $\mbb{R}^* \in N$
 (and $H$ is such a model $N$).
So now
 $$\mbb{R}^* \subseteq
 \mbb{R}^{V(\mbb{R}*)} \subseteq
 \mbb{R}^H =
 \mbb{R}^*,$$
 thus $\mbb{R}^* = \mbb{R}^{V(\mbb{R}^*)}$
 as desired.
\end{proof}

\begin{observation}
If $M(\mbb{R}^*)$ is a symmetric extension of $M$
 as witnessed by $G \subseteq \mbox{Col}(\omega,{<}\delta)$
 and $\delta$ is inaccessible, then
 $$\mbb{R}^* = \mathbb{R}^{M[G]}.$$
\end{observation}

\begin{lemma}
\label{sym_ext_lemma}
Let $\delta$ be a strong limit cardinal.
Let $V(\mbb{R}^*)$ be a symmetric extension of $V$
 as witnessed by
 some $G_1$ that is
 $\mbox{Col}(\omega,{<}\delta)$-generic over $V$.
Then there is a
 $G_2$ that is $\mbox{Col}(\omega,{<}\delta)$-generic
 over \tn{HOD} such that $\tn{HOD}(\mbb{R}^*)$ is a symmetric
 extension of $\tn{HOD}$ as witnessed by $G_2$.
\end{lemma}
\begin{proof}
We will verify the hypothesis of the ``folk''
 Lemma 3.1.5 of \cite{larson_tower_book}.
The easier condition is that $$\delta =
 \mbox{sup}\{ \omega_1^{\tn{HOD}[x]} : x \in \mbb{R}^* \}.$$
This is true because every $\alpha < \delta$ is coded
 to be countable by some $x \in \mbb{R}^*$
 (the generic $G_1$ does this).

The second condition we must verify is that every
 $x \in \mbb{R}^*$ is generic over $\tn{HOD}$
 by some poset in $(V_\delta)^\tn{HOD}$.
Fix $x \in \mbb{R}^*$.
By the generic $G_1$,
 we have that $x$ is generic over $V$
 by $\mbox{Col}(\omega,\lambda)$ for some fixed
 $\lambda < \delta$.
If we let $G_1^- := G_1 \restriction \lambda$,
 we see that
 $x \in V[G_1^-]$.
Note that $\delta$ is a strong limit cardinal
 in $V[G_1^-]$.
We now get that $x$
 is generic over $\tn{HOD}^{V[G_1^-]}$ by a Vop\v{e}nka poset
 $\mathbb{P} \in \tn{HOD}^{V[G_1^-]}$ where
 $|\mathbb{P}| \le (2^{2^\omega})^{V[G_1^-]}$.
Since $\mbox{Col}(\omega,\lambda)$ is homogeneous,
 we have that $\tn{HOD}^{V[G_1^-]}$ is a ground of
 $\tn{HOD}$.
Thus $\tn{HOD}$ is intermediate
 between $\tn{HOD}^{V[G_1^-]}$ and $V[G_1^-]$,
 and so by the proof of the intermediate model theorem
 (see Lemma 15.43 of Jech \cite{jech_book})
 it follows that $x$ is generic over $\tn{HOD}$
 by a poset of size
 \[ (2^{2^\omega})^{V[G_1^-]} < \delta. \qedhere \]
\end{proof}

So here is the main result of this section:

\begin{theorem}
\label{ad_in_ext_down_to_hod}
Assume there is a proper class of
 Woodin cardinals.
Let $A$ be a set of reals in $\tn{HOD}$
 that is universally Baire in $\tn{HOD}$.
Let $\delta$ be an inaccessible cardinal of $V$.
Then the $\mbox{Col}(\omega,{<}\delta)$
 forcing extension of $\tn{HOD}$
 satisfies that $L(A^{**},\mathbb{R})
  \models \tn{AD}$,
  where $A^{**}$ is the reinterpretation
  of $A$ in the forcing extension.
\end{theorem}
\begin{proof}
Let $A^*$ be the reinterpretation of $A$
 in $V$.
Let $G_1$ be $\mbox{Col}(\omega,{<}\delta)$-generic
 over $V$.
Let $A^{**}$ be the reinterpretation of $A^*$
 in $V[G_1]$
 (and note that $A^{**}$ is universally Baire
 in $V[G_1]$).
Since $V$ has a proper class of Woodin cardinals,
 it follows that $V[G_1]$ has a proper class
 of Woodin cardinals, and so
 $L(A^{**},\mathbb{R})^{V[G_1]} \models \tn{AD}$
 (see Fact~\ref{fact_martin_steel_woodin}).

Let $\mbb{R}^*$ be the set of reals in $V[G_1]$.
Note that $A^{**} \in V(\mbb{R}^*)$ and so
 $$L(A^{**},\mathbb{R})^{V(\mbb{R}^*)}
 =  L(A^{**},\mathbb{R})^{V[G_1]}
 \models \tn{AD}.$$
We have that $V(\mbb{R}^*)$ is a symmetric
 extension of $V$.
By Lemma~\ref{sym_ext_lemma},
 the model $\tn{HOD}(\mbb{R}^*)$
 is a symmetric extension of $\tn{HOD}$
 as witnessed by some fixed
 $G_2$ that is $\mbox{Col}(\omega,{<}\delta)$-generic
 over $\tn{HOD}$.
Here is a key:
 the $L(A^{**},\mathbb{R})$ model of
 $V(\mbb{R}^*)$ is the same as the
 $L(A^{**},\mathbb{R})$ model of
 $\tn{HOD}(\mbb{R}^*)$.
Note also that $A^{**}$ is the reinterpretation
 of $A$ in $\tn{HOD}(\mbb{R}^*)$.
Thus
 $L(A^{**},\mathbb{R})^{\tn{HOD}(\mbb{R}^*)} \models \tn{AD}$,
 as desired.
\end{proof}

\section{Local results}

\begin{proposition}
Assume that every $\Sigma^1_2$
 set of reals is universally Baire.
Then this holds in \tn{HOD}.
\end{proposition}
\begin{proof}
It is a fact that
 every $\Sigma^1_2$ set of reals is universally Baire
 iff every set has a sharp (see \cite{feng_ub}).
Thus we want to show that $\tn{HOD}$ satisfies that
 every set of ordinals has a sharp.
So now assume that $x \in \tn{HOD}$ is a set of ordinals.
We will show that $\tn{HOD} \models x^\#$ exists.

Since $x^\#$ exists (in $V$),
 $x^\#$ is OD from $x$, and $x$ is OD,
 it follows that $x^\#$ is OD.
So $x^\# \in \tn{HOD}$, but also
 $(x^\#)^\tn{HOD} = x^\#$ because
 if an inner model thinks a set has a sharp,
 that sharp is computed correctly by the inner model.
Note: An alternate proof is to use Vop\v{e}nka's
 theorem to capture $x^\#$ in a forcing extension of $\tn{HOD}$
 and then note that forcing cannot change whether or not a
 set has a sharp.
\end{proof}

\begin{proposition}\label{prp:pd_correctness}
    Suppose that projective determinacy holds. Then \(\tn{HOD}\) is projectively correct, or in other words,
    \(V_{\omega+1}\cap \textnormal{HOD}\preceq V_{\omega+1}\). In particular, projective determinacy holds in \(\tn{HOD}\).
    \begin{proof}
        Suppose \(x\in V_{\omega+1}\cap \textnormal{HOD}\), and suppose \(A\subseteq V_{\omega+1}\) is a nonempty set that is definable over \(V_{\omega+1}\) from \(x\). We will show that \(A\cap \tn{HOD}\) is nonempty. By the Moschovakis basis theorem
        \cite[6C.6]{moschovakis_book}, \(A\) has an element \(a\) that is \(\Sigma^1_n(x)\) for some \(n < \omega\). Therefore \(a\in A\cap \tn{HOD}\), as desired.
    \end{proof}
\end{proposition}

We will show that if there is a proper class of 
 strong limit cardinals $\lambda$
 such that $V^{\tn{Col}(\omega,{<\lambda})} \models \tn{PD}$,
 then every (set) forcing extension of $V$
 satisfies $\tn{PD}$.

\begin{fact}
For every tree $T$ on
 $(\omega \times \omega) \times \kappa$
 for some ordinal $\kappa$,
 there is a tree $S$ on
 $\omega \times (\omega \times \kappa)$
 such that the relation
 $p[S] = \exists^\mbb{R} p[T]$
 holds in every (set)
 forcing extension of $V$.
We may write
 $S = \exists^\mbb{R} T$.
\end{fact}
\begin{proof}
Simply take $S$ to be $T$,
 but regarded as a tree on
 $\omega \times (\omega \times \kappa)$
 instead of a tree on
 $(\omega \times \omega) \times \kappa$.
\end{proof}

\begin{remark}
\label{projective_bootstrap}
Suppose we have a projective formula $\varphi$
 that takes a single real number as a parameter.
Suppose we want to show that there is a tree $T$
 such that in every (set) forcing extension of $V$,
 $p[T]$ is the set defined by $\varphi$.
We will explain how the only obstacle for doing this is
 the following problem:
 Given a tree $T_a \subseteq {^{<\omega}(\omega \times \tn{Ord})}$,
 we want a tree $T_b \subseteq {^{<\omega}(\omega \times \tn{Ord})}$ such that
 $T_a$ and $T_b$ project to complements in all
 forcing extensions.
The instances of this problem we care about
 cannot be solved in $\zfc$.

Assume that $\varphi$ is written in the form
 $$(\exists y_1 \in \baire)(\forall y_2 \in \baire)...
 (\exists y_n \in \baire)\, \psi(x,y_1,...,y_n)$$
 where $\psi$ is $\Sigma^0_\omega$.
We can always assume this because we can add
 an extra $\exists^\baire$ quantifier at either the
 beginning or end for a variable that is not used.
We will explain a strategy for constructing a
 sequence
 $T_1^\Sigma, T_1^\Pi,
 T_2^\Sigma, T_2^\Pi, ...
 T_n^\Sigma$
 of trees that have certain relationships to each other.
\begin{itemize}
\item First, we want
 $T_1^\Sigma$ to be a tree that projects to
 $$\{ (x,y_1, ..., y_{n-1}) : (\exists y_n \in \baire)\, \psi(x,y_1,...,y_n) \}$$
 in every forcing extension of $V$.
That is, in every forcing extension,
 $T_1^\Sigma$ projects to the set defined by
 the formula $(\exists y_n \in \baire)\, \psi(x,y_1,...,y_n)$.
The existence of $T_1^\Sigma$ is guaranteed
 because $\Sigma^1_1$ sets of reals
 are $\omega$-Suslin.
\item Then we let $T_1^\Pi$ be a tree
 which projects to the complement of $p[T_1^\Sigma]$
 in every forcing extension.
The existence of $T_1^\Pi$ is guaranteed
 by the Shoenfield absoluteness theorem.
\item Then we let $T_2^\Sigma$ be
 $\exists^\mbb{R} T_1^\Pi$.
Note that $T_2^\Sigma$ projects
 to $\exists^\mbb{R} p[T_1^\Pi]$
 in every forcing extension of $V$.
\item Now we want $T_2^\Pi$ to be any tree
 that projects to the complement
 of $p[T_2^\Sigma]$ in every forcing extension
 of $V$.
Here is the first obstacle,
 because $\zfc$ does not guarantee that
 any such tree $T_2^\Pi$ exists.
However, such a tree does exist if we
 assume that every set has a sharp
 (see the Martin-Solovay tree construction
 \cite{kanamori_book}).
It follows that $T_2^\Pi$
 projects to
 $$\{ (x_1,y_1,...,y_{n-2}) : (\forall y_{n-1} \in \baire)(\exists y_n \in \baire)\,
 \psi(x,y_1, ..., y_n) \}$$
 in every (set) forcing extension.
\item Then we let $T_3^\Sigma$ be
 $\exists^\mbb{R} T_2^\Pi$, etc.
\end{itemize}

We proceed like this until we get the final
 tree $T_n^\Sigma$
 that projects to the set defined by $\varphi$
 in every forcing extension.
The key steps are when we go from
 $T_i^\Sigma$ to $T_i^\Pi$,
 which are not guaranteed in $\zfc$.
\end{remark}

The next lemma relies on the following consequence of $\tn{PD}$.
\begin{fact}
\label{pd_implies_uniform_scale_property}
Let $n \in \omega$ be non-zero.
There is a formula $\varphi$ such that $\zfc$ proves
 that $\varphi$ defines a universal
 $\Sigma^1_{2n}$ subset of $\omega$.
Moreover, there is a formula $\psi$ such that
 $\zfc + \tn{PD}$ proves that the tree
 $T := \{ t \in {^{<\omega}(\omega \times \tn{Ord})} : \psi(t) \}$
 projects to the subset of $\baire$ defined by $\varphi$.
Hence, if $\lambda$ is any cardinal and
 $\tn{Col}(\omega,{<}\lambda)$ forces $\tn{PD}$,
 then we can use the homogeneity of $\tn{Col}(\omega,{<}\lambda)$
 to get that $T \in V$.
Hence, in $V$ there exists a tree $T$ such that
 in the $\tn{Col}(\omega,{<}\lambda)$ extension,
 the tree $T$ projects to a universal $\Sigma^1_{2n}$
 subset of $\baire$ (the one defined by $\varphi$).
\end{fact}
\begin{proof}
See the Second Periodicity Theorem in
 \cite{moschovakis_book}.
The statement of that theorem gives us that $\tn{PD}$ implies
 that $\Sigma^1_{2n}$ has the scale property,
 but we need the stronger claim that the scales
 for canonical universal $\Sigma^1_{2n}$ sets can be chosen
 to be defined by the same formulas in all models of
 $\zfc + \tn{PD}$
 (they are defined uniformly).
\end{proof}

\begin{lemma}
\label{using_every_sym_satisfies_pd}
Let $\mbb{P}$ be a poset.
Let $\lambda$ be
 a strong limit cardinal such that
 $|\mbb{P}| < \lambda$.
Suppose that the symmetric extension of $V$
 by $\tn{Col}(\omega,{<}\lambda)$
 satisfies $\tn{PD}$.
Then projective statements are absolute
 between $V$ and generic extensions of $V$
 by $\mbb{P}$,
 and these extensions satisfy $\tn{PD}$.
\end{lemma}
\begin{proof}
We want to show that \textit{every}
 forcing extension of $V$ by $\mbb{P}$
 has some property.
We make the following standard maneuver:
 fix $p \in \mbb{P}$ and let
 $\mbb{P}' := \mbb{P} \restriction p$.
We will show that \textit{some}
 forcing extension of $V$ by $\mbb{P}'$
 satisfies the appropriate property.

Let $\gamma := |\mbb{P}'|$.
Fix a $\mbb{P}$-name $\dot{\mbb{Q}}$ such that
 $\mbb{P}' * \dot{\mbb{Q}}$ is forcing equivalent to
 $\tn{Col}(\omega,\gamma)$.
In fact, $\dot{\mbb{Q}}$ can be taken to be a name for
 $\tn{Col}(\omega,\gamma)$, but that does not matter.

Fix a $G$ that is
 $\tn{Col}(\omega,{<}\lambda)$-generic over $V$.
Let $\mbb{R}^* := \bigcup_{\alpha < \lambda}
 \{ \mbb{R}^{V[G \restriction \alpha]} \}$.
By hypothesis,
 the symmetric extension $V(\mbb{R}^*)$ satisfies PD.
From $G$, extract a set
 $G_2 \in V[G]$ that is $\tn{Col}(\omega,\gamma)$-generic
 over $V$.
Let $\mbb{R}^*_2$ be the set of reals in
 $V[G_2]$.

Now since $V(\mbb{R}^*)$ satisfies PD,
 for each $n \in \omega$,
 the universal $\Sigma^1_n$
 and $\Pi^1_n$ sets of $V(\mbb{R}^*)$
 are the projections of trees $T_n^\Sigma, T_n^\Pi \in V$.
That is, these universal sets are the projections
 of definable trees, and we may pull back
 by homogeneity to get that these trees
 $T_n^\Sigma, T_n^\Pi$ are in $V$
 (see Fact~\ref{pd_implies_uniform_scale_property}
 for more details).
Now in $V(\mbb{R}^*)$ we have that the appropriate
 trees project to complements,
 (which gives us that projective statements are
 absolute between $V$ and $V(\mbb{R}^*)$).
Since $V(\mbb{R}^*_2) \subseteq V(\mbb{R}^*)$,
 these trees also project to complements in $V(\mbb{R}^*_2)$.

Since $\tn{Col}(\omega,\gamma)$
 is forcing equivalent to $\mbb{P}' * \mbb{Q}$
 and $G_2$ is $\tn{Col}(\omega,\gamma)$-generic over $V$,
 fix a $G_1 \in V[G_2]$ that is
 $\mbb{P}'$-generic over $V$.
So now we have
 $$V \subseteq V[G_1] \subseteq V[G_2] \subseteq V[G].$$
Let $\mbb{R}^*_1$ be the set of reals in
 $V[G_1]$.
Since $V(\mbb{R}^*_1) \subseteq V(\mbb{R}^*_2)$,
 the trees described in the paragraph above
 project to complements in
 $V(\mbb{R}^*_1)$.
Hence projective statements are
 absolute between $V$ and $V(\mbb{R}^*_1)$.
Since $\mbb{R}^*_1$ is the set of reals of $V[G_1]$,
 we get that projective statements
 are absolute between $V$ and $V[G_1]$.
As a consequence of this,
 since $V$ satisfies PD, so does $V[G_1]$.
\end{proof}

\begin{theorem}
Suppose that every (set) forcing extension of $V$
 satisfies \tn{PD}.
Then every (set) forcing extension of
 $\tn{HOD}$ satisfies \tn{PD} as well.
\end{theorem}
\begin{proof}
Let $\lambda$ be a strong limit cardinal.
Let $G_1$ be $\tn{Col}(\omega,{<}\lambda)$-generic
 over $V$ and assume that the symmetric
 extension $V(\mbb{R}^*)$ of $V$ by
 $\tn{Col}(\omega,{<}\lambda)$ as witnessed by $G_1$
 satisfies PD.
Note that any model with $\mbb{R}^*$ as its
 set of reals also satisfies $\tn{PD}$.

By Lemma~\ref{sym_ext_lemma},
 there is some $G_2$ that is
 $\tn{Col}(\omega,{<}\lambda)$-generic
 over $V$ such that the symmetric extension
 of $\tn{HOD}$ by $\tn{Col}(\omega,{<}\lambda)$
 as witnessed by $G_2$ has $\mbb{R}^*$
 as its set of reals.
Thus this symmetric extension
 $\tn{HOD}(\mbb{R}^*)$ satisfies PD.
Then by Lemma~\ref{using_every_sym_satisfies_pd},
 every forcing extension of $\tn{HOD}$ by a poset
 of size $<\lambda$ satisfies PD.
\end{proof}

What about $\tn{AD}^{L(\mbb{R})}$?
We will show that if this holds in $V$,
 it need not hold in $\tn{HOD}$.

\begin{lemma}
Assume $V = L(\mbb{R})$
 and $\tn{AD}$ holds.
Then $\tn{HOD}$ satisfies
 that every real is $\tn{OD}^{L(\mbb{R})}$.
\end{lemma}
\begin{proof}
See Theorem 1.9 (b) of
 \cite{steel_cabal_3}.
\end{proof}

So if $V = L(\mathbb{R})$ and $\tn{AD}$ holds,
 then $L(\mathbb{R})^\tn{HOD}$ does not satisfy $\ad$.
The point of the next proposition is that starting from
 such a model we can force the Axiom of Choice to be true
 without changing $L(\mathbb{R})^{\tn{HOD}}$,
 so the $L(\mathbb{R})^{\tn{HOD}}$ of the forcing extension
 will not satisfy $\tn{AD}$.

\begin{proposition}
Assume that $\zf + \ad$ is consistent.
Then there is a model of $\zfc + \ad^{L(\mbb{R})}$ whose
 $\tn{HOD}$ does not satisfy $\ad^{L(\mbb{R})}$.
\end{proposition}
\begin{proof}
Start with a model of $\ad$ + ``$V = L(\mbb{R})$.''
Then force with $\mbb{P} := \tn{Col}(\omega_1,\mbb{R})$ to get the model
 $V[G]$, which satisfies the Axiom of Choice.
Since $V$ and $V[G]$ have the same reals,
 we have $L(\mbb{R})^{V[G]} = L(\mbb{R})$
 so $V[G]$ satisfies
 $\tn{AD}^{L(\mbb{R})}$.

We claim that $\tn{HOD}$ and $\tn{HOD}^{V[G]}$ have
 the same reals.
First, since $L(\mbb{R})$ is a definable inner model
 of $V[G]$, we get that every real in
 $\tn{HOD}^{L(\mbb{R})} = \tn{HOD}^V$
 is in $\tn{HOD}^{V[G]}$.
On the other hand,
 every real in $\tn{HOD}^{V[G]}$ is in $\tn{HOD}$,
 because of the homogeneity of $\mbb{P}$.
This proves the claim.
Thus
 $$L(\mbb{R})^{\tn{HOD}} =
 (L(\mbb{R})^{\tn{HOD}})^{V[G]}.$$

Now by the lemma above,
 the $\tn{HOD}$ of $V[G]$ cannot satisfy $\ad^{L(\mbb{R})}$,
 because $\ad$ implies there is some real that is not in $\tn{HOD}$.
\end{proof}

\begin{remark}
Assume there is a proper class of inaccessible cardinals,
 to make the discussion easier.
Assume that $\ad^{L(\mbb{R})}$ holds in every (set)
 forcing extension of $V$.
We will show that $\ad^{L(\mbb{R})}$ holds in every (set)
 forcing extension of $\tn{HOD}$.
Fix an inaccessible cardinal $\lambda$.
We will mirror part of the proof of
 Theorem~\ref{ad_in_ext_down_to_hod}.
Let $G_1$ be $\tn{Col}(\omega,{<}\lambda)$-generic over $V$.
By hypothesis, $L(\mathbb{R})^{V[G]} \models \ad$.
Let $\mathbb{R}^*$ be the set of reals in $V[G]$.
We have that $V(\mathbb{R}^*)$ is a symmetric extension of $V$.
By Lemma~\ref{sym_ext_lemma},
 the model $\tn{HOD}(\mbb{R}^*)$
 is a symmetric extension of $\tn{HOD}$
 as witnessed by some fixed
 $G_2$ that is $\mbox{Col}(\omega,{<}\lambda)$-generic
 over $\tn{HOD}$.
Since $\mathbb{R}^*$ are the reals of $\tn{HOD}[G_2]$,
 we have $L(\mathbb{R})^{\tn{HOD}[G_2]} \models \ad$.
Unfix $\lambda$.

Since the $\mbox{Col}(\omega,{<}\lambda)$-generic
 extensions of $\tn{HOD}$ all satisfy $\ad^{L(\mathbb{R})}$,
 by the proof of Theorem 5.2 in \cite{feng_ub}
 we get that $\tn{HOD}$ satisfies that
 every set of reals in its own $L(\mbb{R})$
 is universally Baire.
This implies that $\tn{HOD}$ satisfies that
 $\mbb{R}^\#$ exists and is universally Baire
 (see Theorem 3.9 of \cite{feng_ub}).
Now let $G$ be set generic over $\tn{HOD}$.
Also by the proof of Theorem 5.2 in \cite{feng_ub},
 there is a tree $T \in \tn{HOD}$ on $\omega \times \kappa$
 for some ordinal $\kappa$ such that
 $$\tn{HOD} \models \bigg[ \mbb{R}^\# = p[T] \bigg]$$
 and
 $$\tn{HOD}[G] \models \bigg[ \mbb{R}^\# = p[T] \bigg].$$
That is, the $\mbb{R}^\#$ of the forcing extension
 $\tn{HOD}[G]$ is also the projection of $T$.
This is enough to get that the theory of $L(\mbb{R})$
 in both $\tn{HOD}$ and $\tn{HOD}[G]$ is the same:
 $$L(\mbb{R})^\tn{HOD} \equiv L(\mbb{R})^{\tn{HOD}[G]}.$$
Now since some forcing extension of $\tn{HOD}$
 satisfies $\ad^{L(\mbb{R})}$,
 all of them do.
\end{remark}

\section{Downward to $\tn{HOD}$ theorem (main theorem)}

Here is the fundamental $\Sigma^2_1$ Basis Theorem,
 due to Woodin:

\begin{fact}[$\Sigma^2_1$ Basis Theorem]
\label{sigma_2_1_reflection}
Assume $\zf + \tn{AD}^+$.
Let $\mc{S}$ be a non-empty collection of sets of reals
 which is $\Sigma^2_1$ using a fixed set of reals
 $B$ as a parameter.
Let $\vec{\phi}, \vec{\psi}$ be scales on
 $B$ and its complement.
Then there is some $A \in \mc{S}$ that is
 $\Delta^2_1(\vec{\phi}, \vec{\psi})$.
\end{fact}

\begin{remark}
Unfortunately, we think that the proof of the fact
 above is unpublished.
The $\Sigma^2_1$ Basis Theorem, without the
 parameter $B$ is proved in Theorem 7.1 of
 \cite{steel_dm}.
It is claimed there that the
 version with the parameter follows by a similar proof.
However, unfortunately the generalized version
 stated in Lemma 7.2 of \cite{steel_dm} is not right.
There it is claimed that we can find some
 $A \in \mc{S}$ that is $\Delta^2_1(B)$.
We will now show a counterexample to this claimed result:

Assume $\zf + \ad^+ + \ad_\mbb{R} + \tn{DC}$
 and let $B := \mathbb{R} \setminus \tn{HOD}$.
By the axioms we are assuming,
 every set of reals is Suslin
 (see Theorem 0.3 of \cite{larson_ad_book}).
Hence $B$ is Suslin-co-Suslin.
Now let $\mc{S}$ be the set of scales for $B$.
This is a non-empty $\Sigma^2_1$ collection that
  uses $B$ as a parameter.
By the result claimed in Lemma 7.2 of \cite{steel_dm},
 fix a $C \in \mc{S}$ that is $\Delta^2_1(B)$.
Since $C$ is $\Delta^2_1(B)$, it is $\tn{OD}$.
Since $C$ is a scale for $B$,
 there is a tree $T \in \tn{HOD}$ on $\omega \times \tn{Ord}$
 for $B$.
Using the leftmost branch of $T$, we get that
 $B$ intersects $\tn{HOD}$.
This is a contradiction.

The statement of the incorrect Lemma 7.2
 in \cite{steel_dm} is also repeated
 in Theorem A.10 part (3) of
 \cite{sargsyan_msc_book}.
\end{remark}

\begin{lemma}
\label{reinterpretations_are_almost_od}
Let $E \subseteq \baire$ be a set of reals that is both
 $\tn{OD}$ and $\tn{uB}$.
Let $V[G]$ be a (set) forcing extension of $V$
 and let $E^*$ be the reinterpretation of $E$ in $V[G]$.
Then there is an ordinal $\lambda$ such that if we let
 $S := (V_\lambda)^V$, then $E^*$ is definable in $V[G]$
 using only $S$ as a parameter.
\end{lemma}
\begin{proof}
Because $E$ is OD,
 fix an ordinal $\lambda$ such that
 $E$ is definable in $V_\lambda$ without any parameters
 (this can be done using the L\'{e}vy Reflection Principle).
Furthermore, $\lambda$ can be chosen to be arbitrarily large.
So let us assume that $V_\lambda$ contains a tree $T$
 on $\omega \times \lambda$
 such that $T$ projects to $E$ in $V$ (and hence in $S$)
 and it projects to $E^*$
 in $V[G]$.
Let $S := V_\lambda$.
Here is how we can define $E^*$ in $V[G]$ using only
 $S$ as a parameter:
 $E^*$ is the union of the projections of all trees
 $T \in S$ on $\omega \times \lambda$
 such that $T$ projects to $E$ in $S$.
\end{proof}

Recall that $\Gamma^\infty$ is the collection of all
 universally Baire sets of reals.

\begin{lemma}[Tree Stretching]
\label{tree_stretching}
Assume that $\Gamma^\infty$ is
 closed under scales, sharps, and projections.
Let $B \subseteq \baire$ be such that
 $B$ and its complement have scales which are
 \tn{OD} and $\tn{uB}$.
Then $B$
 is \tn{uB} via \tn{OD} trees. 
\end{lemma}
\begin{proof}
Let $\vec{\eta}$ and $\vec{\chi}$ be scales
 for $B$ and its complement respectively
 that are both $\tn{OD}$ and $\tn{uB}$.

Since $B, \vec{\eta}, \vec{\chi}$
 are all uB,
 fix a uB set of reals $D$ such that
 $$B, \vec{\eta}, \vec{\chi} \in L(D,\mbb{R}).$$

We must show that $B$ is $\tn{uB}$ via $\tn{OD}$
 trees.
To do this, fix a cardinal $\kappa$
 and force over $V$
 with $\mbox{Col}(\omega,\kappa)$.
Let $$B^*, D^*, \vec{\eta}^*, \vec{\chi}^*$$
be the reinterpretations
 of $$B,D,\vec{\eta},\vec{\chi}$$
 in $V[G]$.
Using Theorem~\ref{sem_engine} and the fact
 that $\Gamma^\infty$ is
 closed under scales, sharps, and projections,
 fix an elementary embedding
 $$j : L(B,D,\vec{\eta},\vec{\chi},\mathbb{R}) \to
 L(B^*, D^*, \vec{\eta}^*,\vec{\chi}^*,\mathbb{R})^{V[G]}$$
 that maps $B$ to $B^*$,
 $D$ to $D^*$, etc.

By elementarity,
 $\vec{\eta}^*,\vec{\chi}^*$ are scales
 for $B^*$ and its complement (in $V[G]$).
Because $\vec{\eta}$ and $\vec{\chi}$
 are both $\tn{OD}$ (in $V$) and $\tn{uB}$,
 the scales $\vec{\eta}^*$ and $\vec{\chi}^*$
 are both definable in $V[G]$ using some
 fixed $(V_\lambda)^V$ (which we will call $S$)
 as a parameter
 (see Lemma~\ref{reinterpretations_are_almost_od}).
That is,
 $$\vec{\eta}^*, \vec{\chi}^* \in
 \tn{OD}_{\{S\}}^{V[G]}.$$
Recall that every scale for a set $X \subseteq \baire$
 has a tree $T \subseteq {^{<\omega}(\omega \times \tn{Ord})}$
 associated with that scale such that $X = p[T]$.
If a scale is $\tn{OD}_{S}$, the tree associated with it is also
 $\tn{OD}_{S}$.
Thus fix the trees
 $$T^+, T^- \in \tn{HOD}_{\{S\}}^{V[G]}$$
 associated with the scales $\vec{\eta}^*$ and $\vec{\chi}^*$.
So, these trees project to complements (in $V[G]$)
 and $B^* = p[T^+]$.
Then since $\mbox{Col}(\omega,\kappa)$ is homogeneous
 and $\tn{OD}$
 and since $S$ is $\tn{OD}$
  we have that
 $T^+$ is ordinal definable in $V$.
Hence, $T^+ \in \tn{HOD}$.
Similarly,
 $T^- \in \tn{HOD}$.

So we have that $T^+, T^- \in \tn{HOD}$
 project to $B$ and its complement (in $V$),
 and they project to complements
 in $V[G]$.
This shows that $B$ is
 ${{\le} \kappa}$-$\tn{uB}$ via \tn{OD} trees.
Unfixing $\kappa$, we have shown that
 $B$ is $\tn{uB}$ via $\tn{OD}$ trees.
\end{proof}

\begin{theorem}[Main Theorem]
\label{downward_to_hod_without_pc_woodins}
Assume that $\Gamma^\infty$
 is strongly determined and
 closed under scales, sharps, and projections
 (see Definition~\ref{def_enriched}).
Let $C$ be a set of reals such that it and its
 complement have scales which are $\tn{OD}$ and $\tn{uB}$.
Let $\varphi(C)$ be a true
 $(\Sigma^2_1)^{\tn{uB}}$ statement that uses $C$
 as a parameter.
Then $\varphi(C \cap \tn{HOD})$ is true in $\tn{HOD}$.
\end{theorem}
\begin{proof}
By Lemma~\ref{tree_stretching},
 $C$ is $\tn{uB}$ via $\tn{OD}$ trees.
So let $\bar{C} := C \cap \tn{HOD}$.
The set $\bar{C}$ is in $\tn{HOD}$,
 is universally Baire in $\tn{HOD}$,
 and $C$ is the reinterpretation of $\bar{C}$
 in $V$.

Assume that $\varphi$ when written out is
 $$\varphi \Leftrightarrow
 (\exists X \in \tn{uB})\,
 \langle H(\omega_1), X, C, \in \rangle \models \psi.$$
 for some fixed formula $\psi$.
Fix a universally Baire set $A \subseteq \mathbb{R}$
 that witnesses that $\varphi$ is true.
That is,
 $$\langle H(\omega_1), A, C, \in \rangle \models \psi.$$

Let $\vec{\varphi}$ and $\vec{\psi}$ be scales
 for $C$ and its complement which are $\tn{OD}$
 and $\tn{uB}$.
We have that $C$ is universally Baire in $V$,
 so fix a universally Baire set of reals
 $D$ such that
 $$A,C,\vec{\varphi},\vec{\psi}
 \in L(D,\mathbb{R}).$$
Note that since $\vec{\varphi}, \vec{\psi} \in L(D,\mbb{R})$,
 we get that $C$ is Suslin-co-Suslin in $L(D,\mathbb{R})$.
Since $\Gamma^\infty$ is strongly determined,
 the model $L(D,\mathbb{R})$ satisfies $\ad^+$,
 and so by the $\Sigma^2_1$ basis theorem
 (Fact~\ref{sigma_2_1_reflection})
 fix $B,\vec{\eta},\vec{\chi} \in L(D,\mathbb{R})$
 that are all $\Delta^2_1(\vec{\varphi},\vec{\psi})$
 in $L(D,\mathbb{R})$
 such that $\vec{\eta},\vec{\chi}$ are scales for $B$
 and its complement and
 $$\langle H(\omega_1), B, C, \in \rangle \models \psi.$$
In this same invocation of the $\Sigma^2_1$ basis theorem,
 for each $n$, get a $\Delta^2_1(\vec{\varphi},\vec{\psi})$
 scale $\vec{\mu}_n$ for a universal $\Sigma^1_n(B,C)$ set.
These scales will be useful later.
Let $\bar{B} := B \cap \tn{HOD}$
 and also note that
 $\bar{C} = C \cap \tn{HOD}$.

Note that $B,\vec{\eta},\vec{\chi}$ are $\tn{OD}$
 and $\tn{uB}$.
Let us explain.
They are all $\tn{uB}$ because they are all elements
 of $L(D,\mathbb{R})$, $D$ is $\tn{uB}$,
 and the universally Baire sets are closed under sharps.
They are all $\tn{OD}^{L(D,\mathbb{R})}$
 because they are $\Delta^2_1(\vec{\varphi},\vec{\psi})$
 in $L(D,\mathbb{R})$.
The model $L(D,\mathbb{R})$ itself is $\tn{OD}$,
 because it can defined as the unique model
 $L(E,\mathbb{R})$ where $E$ is any $\tn{uB}$
 set of reals that has the same Wadge rank as $D$
 (for this to work, we use that
 $\Gamma^\infty$ is strongly determined).
Since $B,\vec{\eta},\vec{\chi}$ are all
 $\tn{OD}^{L(D,\mathbb{R})}$ and the model
 $L(D,\mathbb{R})$ itself is $\tn{OD}$,
 we get that $B,\vec{\eta},\vec{\chi}$ are all $\tn{OD}$.

To see that $\bar{B}$ is universally Baire
 in $\tn{HOD}$, use Lemma~\ref{tree_stretching}.
That is, since $B$ and its complement
 have scales which are $\tn{OD}$ and $\tn{uB}$
 (as witnessed by
 the scales $\vec{\eta}$ and $\vec{\chi}$),
 $B$ is uB via OD trees.

Finally, it suffices to show that
 $$\langle H(\omega_1)^{\tn{HOD}},
 \bar{B}, \bar{C}, \in \rangle
 \preceq \langle H(\omega_1), B, C, \in \rangle.$$
By the Tarski-Vaught Test,
 it suffices to show that every non-empty
 $\Sigma^1_n(B,C,x)$ set of reals intersects $\tn{HOD}$
 for every $x \in \baire \cap \tn{HOD}$.
Earlier we found, for each $n \in \omega$,
 a $\Delta^2_1(\vec{\varphi},\vec{\psi})$
 scale $\vec{\mu}_n$ for a universal $\Sigma^1_n(B,C)$ set.
Since $\vec{\varphi},\vec{\psi}$ are both OD,
 the scale $\vec{\mu}_n$ is also OD.
This is enough.
\end{proof}

Because a proper class of Woodin cardinals implies
 that $\Gamma^\infty$ is strongly strongly
 determined and closed under scales, sharps, and projections,
 we have the following:
\begin{theorem}
\label{downwards_with_strongly_od_param_with_woodins}
Assume there is a proper class of Woodin
 cardinals.
Let $C$ be a set of reals such that it
 and its complement have scales
 which are $\tn{OD}$ and $\tn{uB}$.
Let $\varphi(C)$ be a true
 $(\Sigma^2_1)^{\tn{uB}}$ statement
 that uses $C$ as a parameter.
Then $\varphi(C \cap \tn{HOD})$ is true in $\tn{HOD}$.
\end{theorem}

From the result above we get the following:
\begin{theorem}
\label{intro_main_downwards_result_no_param}
Assume there is a proper class of Woodin cardinals.
Let $\varphi$ be a true $(\Sigma^2_1)^\tn{uB}$ statement.
Then $\tn{HOD}$ satisfies $\varphi$.
\end{theorem}

\section{Having scales that are \tn{OD} and \tn{uB}}

In Theorem~\ref{downward_to_hod_without_pc_woodins},
 we have a special assumption on our parameter $C$:
 we assumed that $C$ and its complement have scales
 which were \textit{both} OD and uB at the same time.
The point of this section is to discuss that hypothesis.
Specifically, it is important we investigate the
 relationship between that concept and the perhaps
 strictly weaker concept of being \tn{uB} via \tn{OD} trees.

\begin{remark}
The following hold:
\begin{itemize}

\item
If a set of reals
 and its complement have scales which are
 \tn{OD} and \tn{uB}
 and there is a proper class of Woodin
 cardinals, then it is uB via OD trees
 (Lemma~\ref{tree_stretching}).

\item
If a set of reals either
 has a scale which is \tn{OD} and \tn{uB}
 or if the set is \tn{uB} via \tn{OD} trees,
 then it is both \tn{OD} and \tn{uB}.

\item
Let $A$ be a set of reals that is
 uB via $\tn{OD}$ trees.
The following hold:
\begin{itemize}
\item $A$ has an OD scale.
\item If in addition there is a proper class of
 Woodin cardinals, then $A$ has a $\tn{uB}$ scale.
\end{itemize}
\end{itemize}
\end{remark}

So if there is a proper class of Woodin cardinals
 and $A$ is a set of reals that is
 uB via $\tn{OD}$ trees,
 then $A$ has an OD scale and a uB scale
 (but we are not claiming that it has a scale
 that is \textit{both} OD and uB).

\begin{remark}
If a set of reals $A$ is OD and uB,
 it need not be uB via \tn{OD} trees.

That is, suppose that there are only countably
 many reals in $\tn{HOD}$.
Let $A$ be the set of reals that are
 \textit{not} in HOD.
Then $A$ is certainly OD.
It is also uB, because it is
 co-countable.
On the other hand,
 A does not have an OD scale,
 because if it did then some element
 of A would be in HOD.
Since A does not have an OD scale,
 it must not be uB via \tn{OD} trees.
\end{remark}

We ask the following:
\begin{question}
Assume there is a proper class of Woodin cardinals.
Let $C$ be a set of reals that is uB via \tn{OD} trees.
Must $C$ and its complement have scales
 which are both $\tn{OD}$ and $\tn{uB}$?
\end{question}

A positive answer to this conjecture would give us
 our downwards-to-HOD theorem with a weaker
 hypothesis on the parameter.

The following theorem explains the connection
 between the stronger downwards absoluteness to $\tn{HOD}$ conjecture
 and the closure properties of the universally Baire sets
 of $\tn{HOD}$:

\begin{theorem}
\label{conj_down_thm_iff_ub_closed_in_hod}
Assume there is a proper class of Woodin cardinals.
The following are equivalent:
\begin{itemize}
\item[1)] $\tn{HOD}$ satisfies that its universally Baire sets
 are projectively closed, closed under scales,
 and closed under sharps.
\item[2)] If $C$ is any set of reals that is $\tn{uB}$
 via $\tn{OD}$ trees, then $C$ and its complement have
  scales which are $\tn{OD}$ and $\tn{uB}$.
\item[3)]
If $C$ is any set of reals that is $\tn{uB}$
 via $\tn{OD}$ trees, then
 $(\Sigma^2_1)^{\tn{uB}}$ statements which use $C$
 as a parameter are downwards absolute from $V$ to $\tn{HOD}$.
\end{itemize}
\end{theorem}
\begin{proof}
1) $\Rightarrow$ 2):
Fix a set of reals $C$ that is uB via \tn{OD} trees.
We will show that $C$ has an OD and uB scale in $V$.
A similar argument shows that
 the \textit{complement} of $C$
 has an OD and uB scale in $V$.
Fix a set of reals $\bar{C} \in \tn{HOD}$, that is
 uB in HOD, such that $C$ is the canonical
 reinterpretation of $\bar{C}$ in $V$.
Let $\vec{\eta} \in \tn{HOD}$ be such that
 $\tn{HOD}$ satisfies that $\vec{\eta}$
 is a uB scale for $C$.
The existence of $\vec{\eta}$ follows by 1).
Let $\vec{\eta}^*$ be the canonical reinterpretation
 of $\vec{\eta}$ in $V$.
Now by 1), since the uB sets of HOD are
 sufficiently closed,
 there is an elementary embedding
 $$j : L(\mathbb{R},\bar{C},\vec{\eta})^\tn{HOD}
 \to L(\mathbb{R}, C, \vec{\eta}^*)$$
 such that $j(\bar{C}) = C$ and
 $j(\vec{\eta}) = \vec{\eta}^*$.
See Proposition~\ref{upwards_theorem_enriched}.
Hence $\vec{\eta}^*$ really is a scale
 for $C$ in $V$.
Also, $\vec{\eta}^*$ is OD and uB,
 because it is the canonical reinterpretation
 of a uB set $\vec{\eta}$ of HOD.

2) $\Rightarrow$ 3):
This follows from our main result
 (see Theorem~\ref{downwards_with_strongly_od_param_with_woodins}).

3) $\Rightarrow$ 1):
Assume 3).
We will first show that the uB sets of HOD
 are closed under scales.
Let $\bar{C}$ be an arbitrary uB set of reals of HOD.
Let $C$ be the canonical reinterpretation
 of $\bar{C}$ in $V$.
So $C$ is uB via $\tn{OD}$ trees.
The true statement $\varphi(C)$ that there exists a uB
 scale for $C$ is $(\Sigma^2_1)^{\tn{uB}}$.
Thus by 3),
 $\tn{HOD}$ satisfies $\varphi(\bar{C})$.
In other words,
 $\tn{HOD}$ satisfies that there exists
 a uB scale for $\vec{C}$,
 which is what we wanted to show.

Showing projective closure is similar.
Let $\bar{C} \subseteq \baire \times \baire$
 be an arbitrary uB set of reals of HOD.
Let $C$ be the canonical reinterpretation
 of $\bar{C}$ in $V$.
Let $\varphi(C)$ be the statement that there exists
 a $\tn{uB}$ set of reals $B \subseteq \baire$
 such that $B$ is the projection of $C$ onto the $x$-axis.
The statement $\varphi(C)$ is $(\Sigma^2_1)^{\tn{uB}}$,
 and by 3) it reflects into $\tn{HOD}$.

Showing closure under sharps is also similar.
Let $\bar{C} \subseteq \baire \times \baire$
 be an arbitrary uB set of reals of HOD.
Let $C$ be the canonical reinterpretation
 of $\bar{C}$ in $V$.
Let $\varphi(C)$ be the statement that there exists
 a $\tn{uB}$ set of reals $B$ such that
 $B = (C,\mbb{R})^\#$.
Here it is important that whether or not
 $B = (C,\mbb{R})^\#$ can be computed in
 $\langle H(\omega_1), B, C, \in \rangle$.
The statement $\varphi(C)$ is thus $(\Sigma^2_1)^{\tn{uB}}$,
 and by 3) it reflects into $\tn{HOD}$.
\end{proof}

The proof of our
 Theorem~\ref{downward_to_hod_without_pc_woodins}
 sheds some light onto the situation:

 \begin{proposition}
Assume that $\Gamma^\infty$
 is strongly determined and
 closed under scales, sharps, and projections
 (see Definition~\ref{def_enriched}).
Let $\Lambda$ be the collection of sets of reals
 $C$ such that both $C$ and its complement
 have scales which are \tn{OD} and \tn{uB}.
Then $\Lambda$ has the scale property.
\end{proposition}
\begin{proof}
Fix $C \in \Lambda$.
Let $\vec{\varphi},\vec{\psi}$ be scales for $C$ and its
 complement respectively that are both
 $\tn{OD}$ and $\tn{uB}$.
By symmetry, it suffices to show that $C$
 has a scale in $\Lambda$.
 
Let $D$ be a uB set of reals such that
 $$C,\vec{\varphi},\vec{\psi} \in L(D,\mathbb{R}).$$
Within $L(D,\mathbb{R})$ use the $\Sigma^2_1$ basis theorem
 to get a scale $B$ for $C$ that is
 $\Delta^2_1(\vec{\varphi},\vec{\psi})$.
Moreover, in this same invocation
 of the $\Sigma^2_1$ basis theorem,
 get scales $B^+$ and $B^-$ for $B$ and its complement
 that are both $\Delta^2_1(\vec{\varphi},\vec{\psi})$
 in $L(D,\mbb{R})$.
We will show that $B \in \Lambda$.

Since $B^+$ and $B^-$ are in $L(D,\mbb{R})$
 and $\Gamma^\infty$ is closed under sharps,
 we have that $B^+$ and $B^-$ are $\tn{uB}$.
Also
 we have that the model $L(D,\mbb{R})$ is OD,
 because it is the unqiue model of the form $L(E,\mbb{R})$
 where $E$ is any $\tn{uB}$ set of reals of the same
 Wadge rank as $D$
 (for this we use that $\Gamma^\infty$
 is strongly determined).
Combining the fact that $L(D,\mbb{R})$ is OD
 with the fact that
 $\vec{\varphi}$ and $\vec{\psi}$ are OD
 we have that $B^+$ and $B^-$ are OD.
Hence $B^+$ and $B^-$ are both OD and uB,
 which shows that $B$ and its complement have
 scales which are $\tn{OD}$ and $\tn{uB}$.
\end{proof}

\section{The $\Omega$ Conjecture}

Here is what our main result
 (Theorem~\ref{downward_to_hod_without_pc_woodins})
 tells us about $\Omega$-logic:

\begin{proposition}
\label{omega_theorems_go_down}
Assume there is a proper class of Woodin cardinals.
Then $\Omega$-theorems are downwards absolute from
 $V$ to $\tn{HOD}$.
That is,
 let $\varphi$ be a formula in the language of set theory
 and let $x \in \mbb{R} \cap \tn{HOD}$.
Let $T$ be any theory such that
 $T \in \tn{HOD}$.
Then if
 $$T \proves_\Omega \varphi(x),$$
 then
 $$\tn{HOD} \models \bigg[
 T \proves_\Omega \varphi(x)
 \bigg].$$
\end{proposition}
\begin{proof}
The statement
 ``$T \proves_\Omega \varphi(x)$''
 is $(\Sigma^2_1)^{\tn{uB}}$
 using $x$ and $T$ as parameters.
Note that $T$ can be coded by a real number in HOD.
So by Theorem~\ref{downward_to_hod_without_pc_woodins}
 the result follows.
\end{proof}

We will show two interesting applications of
 Proposition~\ref{omega_theorems_go_down}.
For the first application,
 we need a quick proposition
 on $\Omega$-validities
 (that does not use any Woodin cardinals):

\begin{proposition}
\label{omega_valid_go_up}
$\Omega$-validities are upward
 absolute from $\tn{HOD}$ to $V$.
In other words,
 $\Omega$-satisfaction is downwards absolute
 from $V$ to $\tn{HOD}$.
\end{proposition}
\begin{proof}
Let $\varphi$ be a statement that is
 $\Omega$-satisfiable in $V$.
We will show that $\varphi$ is
 $\Omega$-satisfiable in $\tn{HOD}$.
Fix a poset $\mbb{P} \in V$
 and a $G$ that is $\mbb{P}$-generic over
 $V$ such that
 $$(V_\alpha)^{V[G]} \models \varphi$$
 for some fixed ordinal $\alpha$.

Let $\lambda > \max\{ \alpha,|\mbb{P}| \}$
 be a strong limit cardinal.
Let $H \in V$ be (set) generic over $\tn{HOD}$
 such that $V_\lambda \subseteq \tn{HOD}[H]$.
Since $\tn{HOD}[H]$ and $V$ agree up to $\lambda$,
 we get that
 $G$ is $\mbb{P}$-generic over $\tn{HOD}[H]$
 and the models $\tn{HOD}[H][G]$ and $V[G]$ agree up to $\lambda$.
Thus
 $$(V_\alpha)^{\tn{HOD}[H][G]} \models \varphi.$$
Thus, $\tn{HOD}$ satisfies that $\varphi$
 is $\Omega$-satisfiable.
\end{proof}

\begin{cor}
\label{omega_conj_goes_down}
Assume $V$
 has a proper class of Woodin cardinals.
If the $\Omega$ conjecture holds in $V$,
 then it holds in $\tn{HOD}$.
\end{cor}
\begin{proof}
The following diagram shows the
 overall structure of this proof:
$$
\xymatrix{
 \mbox{$\Omega$-theorems of $V$} \ar[d] \ar@{--}[r] &
 \mbox{$\Omega$-validities of $V$} \\
 \mbox{$\Omega$-theorems of $\tn{HOD}$} &
 \mbox{$\Omega$-validities of $\tn{HOD}$ \ar[u]}
}$$
The key is that $\Omega$-validities about $\tn{HOD}$
 can be converted into $\Omega$-validities about $V$,
 and $\Omega$-theorems about $V$ can be converted into
 $\Omega$-theorems about $\tn{HOD}$.

We will show something stronger.
Let $T \in \tn{HOD}$ be a theory
 and let $x \in \mbb{R} \cap \tn{HOD}$.
Let $\varphi$ be a formula and
 suppose that
 $$\tn{HOD} \models \bigg[
 T \models_\Omega \varphi(x)
 \bigg].$$
Then by Proposition~\ref{omega_valid_go_up}
 we have
 $$T \models_\Omega \varphi(x).$$
Combining this with the fact that $V$
 satisfies the $\Omega$ conjecture,
 we get that
 $$T \proves_\Omega \varphi(x).$$
Then by Proposition~\ref{omega_theorems_go_down}
 we get that
\[\tn{HOD} \models \bigg[
 T \proves_\Omega \varphi(x)
 \bigg].\]
Note that we are not claiming that $\tn{HOD}$
 has a proper class of Woodin cardinals,
 only that every $\Omega$-valid statement
 in $\tn{HOD}$ is an $\Omega$-theorem in $\tn{HOD}$.
\end{proof}

Unfortunately, even if \textit{both}
 $V$ and $\tn{HOD}$ have a proper class of
 Woodin cardinals and $\tn{HOD}$
 satisfies the $\Omega$ conjecture,
 it is not clear to us if it is possible to prove that
 $V$ satisfies the
 $\Omega$ conjecture.

The second interesting application of
 Proposition~\ref{omega_theorems_go_down}
 is $V$ being ``$\varphi$-closed''
 for various large cardinal axioms
 $(\exists x) \varphi(x)$.
We will review some of the framework
 developed in Woodin \cite{woodin_ch2}:

\begin{definition}
Let $\varphi$ be a formula.
The statement $(\exists x)\, \varphi(x)$
is called a \textit{large cardinal axiom}
if $\varphi(x)$ is a
$\Sigma_2$-formula and, as a theorem of $\zfc$,
if $\kappa$ is a cardinal such that
$V \models \varphi(\kappa)$,
then $\kappa$ is uncountable,
inaccessible, and for every
complete Boolean algebra $\mbb{B}$
of cardinality $< \kappa$ we have
$V^\mbb{B} \models \varphi(\kappa)$.
\end{definition}

\begin{definition}
Suppose that
 $(\exists x) \varphi(x)$ is a large
 cardinal axiom.
Then $V$ is called $\varphi$-\textit{closed}
 iff for every set $X$,
 there exists a transitive set $M$
 and some $\kappa \in \tn{Ord} \cap M$
 such that
 $M \models \zfc$,
 $\mbox{Rank}(X) < \kappa$,
 and
 $M \models \varphi(\kappa)$.
\end{definition}

Here is a claim made in
 Woodin \cite{woodin_ch2},
 but we include a proof here
 for completeness:

\begin{proposition}
\label{woodin_ch2_paper_claim}
Assume there is a proper class of
 Woodin cardinals.
Then the following are equivalent:
\begin{itemize}
\item[1)] The $\Omega$ conjecture holds.
\item[2)] Suppose $(\exists x) \varphi(x)$
 is a large cardinal axiom such that
 $V$ is $\varphi$-closed.
Then
 $$\tn{ZFC} \proves_\Omega
 \mbox{``$V$ is $\varphi$-closed''.}$$
\end{itemize}
\end{proposition}
\begin{proof}
2) $\Rightarrow$ 1):
Now fix an $\Omega$-valid $\Pi_2$-statement $\psi$, and let us show $\psi$ is $\Omega$-provable. Let $\varphi(\kappa)$ be the large cardinal axiom stating that
$\kappa$ is inaccessible and $\psi$ holds in every small generic extension of $V_\kappa$. Then V is $\varphi$-closed since for any inaccessible $\kappa$, any model $M$ of ZFC containing $V_\kappa$ satisfies $\varphi(\kappa)$ by the $\Omega$-validity of $\psi$.
By 2), ``$V$ is $\varphi$-closed'' is $\Omega$-provable.
Clearly $\text{ZFC}\vdash $ [``$V$ is $\varphi$-closed''$\to$ $\psi$] (because $\psi$ is $\Pi_2$).
So $\psi$ is $\Omega$-provable.

1) $\Rightarrow$ 2):
Let $\varphi$ be a large cardinal axiom such that V is $\varphi$-closed. We claim that the statement ``V is $\varphi$-closed'' is an $\Omega$-valid
 sentence.
It follows from this and the
 $\Omega$ conjecture that
 ``V is $\varphi$-closed''
 is $\Omega$-provable.

To see this claim, fix a complete Boolean algebra $\mathbb B$,
 a $G$ that is $\mbb{B}$-generic over $V$,
 and an ordinal $\alpha$ such that $V_\alpha^{V[G]} \models \zfc$.
Let $\lambda$ be an inaccessible cardinal such that
 $\mbb{B} \in V_\lambda$.
We must show
  $V_\alpha^{V[G]}$ is $\varphi$-closed.
Fix a set $X \in V_\alpha^{V[G]}$.
We must find a transitive set $\bar{M} \in V_\alpha^{V[G]}$
 and an ordinal $\bar{\kappa} \in \bar{M}$ such that
 $\mbox{Rank}(X) < \bar{\kappa}$ and
 $\bar{M} \models \varphi(\bar{\kappa})$.
Fix a transitive model $M$ of $\zfc$
 and some $\kappa \in \tn{Ord} \cap M$ such that
 $X \cup V_\lambda \in M$, $\mbox{Rank}(X \cup V_\lambda) < \kappa$,
  and $M \models \varphi(\kappa)$
 (so $\kappa$ is inaccessible in $M$).
It might seem like we have a problem
 because there is no reason why
 $M \in V_\alpha^{V[G]}$ must hold,
 but we will deal with that soon.

Now $M[G] \models \varphi(\kappa)$
 because $\varphi(\kappa)$
 is invariant under small forcing by the definition of a large cardinal axiom.
In the forcing extension $V[G]$, the model $M[G]$
 contains $X$
 and satisfies $\varphi(\kappa)$.
Here is the final trick.
Let $S \prec M[G]$ be an elementary substructure of
 $M[G]$ such that $X \subseteq S$,
 $\kappa \in S$,
 and $|S| = \mbox{max}\{|X|, \omega\}$.
Let $\bar{M}$ be the transitive collapse of $S$.
Since $X$ is transitive and $X \subseteq S$,
 we have $X \in \bar{M}$.
Let $\bar{\kappa}$ be the image of $\kappa$ under the collapsing map.
By elementarity,
 $\bar{M} \models \varphi(\bar{\kappa})$.
Also, $\mbox{Rank}(X) < \bar{\kappa}$
 by elementarity.
Finally, because
 $|\bar{M}| = \mbox{max}\{|X|,\omega\} < \alpha$
 and since $V_\alpha^{V[G]} \models \zfc$
 this implies $\bar{M} \in V_\alpha^{V[G]}$.
This is what we wanted to show.
That is,
 $V[G]$ is $\varphi$-closed.

Note: There is another way to show 1) $\Rightarrow$ 2).
The key is to show that every forcing extension of $V$
 satisfies ``V is $\varphi$-closed''.
Note that it is clear that the statement
``V is $\varphi$-closed''
 is $\Pi_2$.
So every forcing extension satisfies that
 if $V_\alpha \models \zfc$,
 then $V_\alpha \models $``$V$ is $\varphi$-closed''.
This is because $\Pi_2$ statements are downwards absolute
 to models of the form $V_\alpha$ that are ``beth closed.''
\end{proof}

Now we can get the promised
 second interesting consequence of
 Proposition~\ref{omega_theorems_go_down}:

\begin{cor}
Assume there is a proper class of
 Woodin cardinals.
Assume that the $\Omega$ conjecture holds.
Let $(\exists x)\, \varphi(x)$
 be a large cardinal axiom
 and assume that $V$ is
 $\varphi$-closed.
Then $\tn{HOD}$ is $\varphi$-closed.
\end{cor}
\begin{proof}
By Proposition~\ref{woodin_ch2_paper_claim},
 $$\zfc \proves_\Omega
 \mbox{``$V$ is $\varphi$-closed''}.$$
Then By Proposition~\ref{omega_theorems_go_down}
 we have
 $$\tn{HOD} \models \bigg[
 \zfc \proves_\Omega
 \mbox{``$V$ is $\varphi$-closed''}
 \bigg].$$
By the soundness of $\Omega$-logic
 in $\tn{HOD}$ we have
 $$\tn{HOD} \models \bigg[
 \zfc \models_\Omega
 \mbox{``$V$ is $\varphi$-closed''}
 \bigg],$$
 and so in particular
  \[\tn{HOD} \models \bigg[
 \mbox{``$V$ is $\varphi$-closed''}
 \bigg].\qedhere\]
\end{proof}

\begin{definition}
A cardinal $\kappa$ is \textit{huge}
 iff there is an elementary embedding
 $j : V \to M$
 of $V$ into some inner model $M$,
 where the critical point of $j$ is $\kappa$,
 such that $^{j(\kappa)}M \subseteq M$.
\end{definition}

Every huge cardinal is superstrong,
 and hence is a Woodin limit of Woodin cardinals.
However, the existence of a huge cardinal
 does not imply the existence of a proper class
 of Woodin cardinals
 (but the existence of an \textit{extendible}
 cardinal does imply the existence of a proper
 class of Woodin cardinals).

Here is another consequence of
 Proposition~\ref{omega_theorems_go_down},
 similar to the previous corollary:

\begin{cor}
\label{omega_huge_go_down}
Assume there is a proper class of huge cardinals
 and the $\Omega$ conjecture holds.
Then $\tn{HOD}$ satisfies that every set belongs to
 an inner model with a huge cardinal.
\end{cor}
\begin{proof}
Let $\varphi$ be the following statement:
For every set $X$,
 there is a transitive model $\bar{M}$ of $\zfc$
 such that $X \cup \{\omega_1\} \in \bar{M}$ and
 there are $\bar{\kappa}, \bar{\mu} \in \tn{Ord} \cap \bar{M}$ such that
 $\mbox{Rank}(X) < \bar{\kappa} < \bar{\mu}$,
 $\bar{M} \models \bar{\kappa}$ is huge, and
 $\bar{M} \models \bar{\mu}$ is measurable.
We will show that $\zfc \models_\Omega \varphi$.
Since every forcing extension of $V$ has a proper class
 of huge cardinals, it suffices to show that a proper class
 of huge cardinals implies that $V_\alpha \models \varphi$
 for every $\alpha$ such that $V_\alpha \models \zfc$.

We will do that now (working in any forcing extension
 of the original universe).
Fix an ordinal $\alpha$ such that $V_\alpha \models \zfc$.
Fix $X \in V_\alpha$.
Since there is a proper class of huge cardinals,
 we can easily get that there is a transitive model
 $M$ of $\zfc$ such that
 $X \cup \{\omega_1\} \in M$ and we can get
 $\kappa,\mu \in M$ such that
 $\mbox{Rank}(X) < \kappa < \mu$,
 $M \models \kappa$ is huge, and
 $M \models \mu$ is measurable.
It might not be that $M \in V_\alpha$.
However, we now take an elementary substructure of $M$
 that contains $X \cup \omega_1 \cup \{\omega_1, \kappa, \mu\}$
 and then take its transitive collapse to get
 a transitive model $\bar{M}$ of $\zfc$.
If we let $\bar{\kappa}$ and $\bar{\mu}$ be the images
 of $\kappa$ and $\mu$ under the collapse,
 we see that $\langle \bar{M}, \bar{\kappa}, \bar{\mu} \rangle$
 is as desired.
Unfixing the forcing extension we are working in,
 we have just shown $\zfc \models_\Omega \varphi$.

Since we are assuming the $\Omega$ conjecture and
 $\zfc \models_\Omega \varphi$, we have
 $\zfc \proves_\Omega \varphi$.
Now we use Proposition~\ref{omega_theorems_go_down}
 to get that
 $\tn{HOD} \models ( \zfc \proves_\Omega \varphi )$.
In particular
 $\tn{HOD} \models ( \zfc \models_\Omega \varphi )$.

Fix $X \in \tn{HOD}$.
We will show that there is an inner model of $\tn{HOD}$
 which contains $X$ and a huge cardinal.
Fix an inaccessible cardinal $\alpha > \mbox{Rank}(X)$.
We have $V_\alpha^{\tn{HOD}} \models \zfc$.
Since $\tn{HOD} \models (\zfc \models_\Omega \varphi)$,
 we have $V_\alpha^{\tn{HOD}} \models \varphi$.
Using that, fix a transitive model $\bar{M} \in V_\alpha^\tn{HOD}$
 of $\zfc$ such that
 $$X, \omega_1^{\tn{HOD}} \in \bar{M}$$
 and $\bar{M}$ contains ordinals $\bar{\kappa}$ and $\bar{\mu}$ such that
 $\bar{M} \models \bar{\kappa}$ is huge and
 $\bar{M} \models \bar{\mu}$ is measurable.
Let $\mc{U} \in \bar{M}$ be such that
 $\bar{M} \models \mc{U}$ is a normal measure on $\bar{\mu}$.
The structure $\langle \bar{M}, \mc{U} \rangle$ can be linearly iterated
 \tn{Ord} times by Corollary 2.8 of \cite{steel_intro_iterated_ult}
 because $\omega_1^\tn{HOD} \in \bar{M}$.
This linear iteration turns $\bar{M}$ into an inner model $N$.
In this inner model $N$, we have
 $X \in N$ and
 $N \models \bar{\kappa}$ is huge,
 because the critical points used in the iteration are all above
 $\mbox{Rank}(X)$ and $\bar{\kappa}$.
\end{proof}

\section{$\Omega$ proofs that reals are in $\tn{HOD}$}

In this section we will investigate
 some technology surrounding
 $$(\forall^\baire)(\Sigma^2_1)^{\tn{uB}}$$
 absoluteness and how it relates to \tn{HOD}.
The reader interested in this topic
 might want to read \cite{wilson_ub_and_gen_abs},
 which discusses both
 lightface $(\forall^\baire)(\Sigma^2_1)^{\tn{uB}}$ generic absoluteness
 and boldface
 one-step and two-step
 $(\forall^\baire)(\Sigma^2_1)^{\tn{uB}}$ generic absoluteness.

Let us make the following definition
 to ease the discussion:
\begin{definition}
A real number $x$ is called an
 $\Omega$-real iff
 $$\zfc + \mbox{``there exists a Woodin cardinal''} \proves_\Omega (x \in \tn{HOD}).
 $$
\end{definition}

If there exists a Woodin cardinal
 and a proper class of inaccessible cardinals,
 then every $\Omega$-real is in \tn{HOD}.
This is by the soundness of $\Omega$-logic
 and the fact that every
 real that is in $\tn{HOD}^{V_\alpha}$
 for some ordinal $\alpha$
 is also in $\tn{HOD}$.
It is also true that if a real
 is in a countable premouse that has a
 universally Baire $\omega_1$-iteration
 strategy, then that real is an $\Omega$-real
 (see Proposition~\ref{reals_in_mice_are_omega_reals}).

On the other hand,
 in Theorem~\ref{ult_l_implies_exery_real_is_omega}
 we will prove that the
 axiom ``$V$ = Ultimate $L$''
 implies that every real is an
 $\Omega$-real.
This is due to Woodin,
 but we provide a proof here for
 completeness.
The way the proof works is to first
 show that the axiom ``$V$= Ultimate $L$''
 implies for each real $x$,
 there is a universally Baire set of reals $A$
 such that $x$ is OD in $L(A,\mbb{R})$
 (see Lemma~\ref{ultl_reals_in_hod_lar}).
Then we show that the conclusion
 of the last sentence
 for any fixed real $x$
 implies that $x$ is an $\Omega$-real
 (see Lemma~\ref{od_in_lar_implies_omega_real}).

Proposition~\ref{phi2_implies_super_abs_fail}
 says that if there is a Woodin cardinal
 and every real is an $\Omega$-real, then
 generic $(\forall^\baire)(\Sigma^2_1)^{\tn{uB}}$
 absoluteness fails.
This then gives us 
 Corollary~\ref{lc_con_with_ultl_cannot_imply_super_abs} from the introduction.
That is, large cardinals consistent
 with the axiom ``$V$ = Ultimate $L$'' cannot imply
 $(\forall^\baire)(\Sigma^2_1)^{\tn{uB}}$ absoluteness
 between $\tn{HOD}$ and $V$.

We will need the following lemma about determinacy theory
 for the lemma that follows it.
This fact is well-known, but we are including a proof for completeness.
The proof uses the connection of being
 ${<}\Theta$-$\tn{uB}$ and ${<}\Theta$-homogeneously Suslin.
We recommend \cite{steel_dm}
 for the interplay between
 the property of being universal Baire
 and the property of being $\infty$-homogeneously Suslin.
The Moschovakis coding lemma can be found as
 Lemma 5 of
 \cite{a_coding_lemma_paper}. 
We also refer the reader to
 \cite{martin_solovay_basis_theorem}
 and \cite{kanamori_book} for the Martin-Solovay tree.
 \begin{lemma}
\label{od_scs_is_less_than_theta_ub_in_hod}
Assume $\ad + \tn{DC}_\mbb{R}$.
Suppose $B \subseteq \baire$ is $\Delta^2_1$.
Then $B \cap \tn{HOD}$ is ${<}\Theta$-\tn{uB} in $\tn{HOD}$.
\end{lemma}
\begin{proof}
Since $B$ and its complement are Suslin,
 $B$ is homogeneously Suslin
 by a theorem of Martin
 (see \cite{tree_of_m_scale_is_homog}).
We can show that in fact $B$ has an OD
 homogeneous tree $(T,\vec{\mu})$.
To see this, first recall that $\ad^+$ implies that
 $\Sigma^2_1$ has the scale property
 (this is discussed after Theorem 7.1 of \cite{steel_dm}).
So fix scales $\vec{\varphi}$ and $\vec{\psi}$ for
 $B$ and its complement that are $\Sigma^2_1$
 (and hence $\tn{OD}$).
Then by the $\Sigma^2_1$ Basis Theorem
 (see Fact \ref{sigma_2_1_reflection}), there is a
 $\Delta^2_1(\vec{\varphi},\vec{\psi})$ set $E$
 (which is hence $\tn{OD}$)
 that codes the
 homogeneity system for $B$.
More precisely, $E$ codes via transitive collapse a
 model $L_\alpha(D,\mbb{R})$ satisfying $\zf^-$
 such that there is a homogeneity system for $B$ that is
 definable from $D$ over $L_\alpha(D,\mbb{R})$;
 note that for all $\gamma < \alpha$, we have
 $P(\gamma) \in L_\alpha(D,\mbb{R})$ by the
 Moschovakis coding lemma, so the measures in the
 homogeneity system really are measures.
The homogeneous tree is $\tn{OD}_E$ and hence
 $\tn{OD}_B$ and hence OD.

It is true that $\Theta$ is the limit of weak partition cardinals
 (see \cite{kechris_ad}).
Now let $\delta < \Theta$ be a weak partition cardinal
 above $|T|$, and let $S$ be the Martin-Solovay tree
 associated to $(T,\vec{\mu},\delta)$.
By the proof of Lemma 3.6 of \cite{jackson_intro},
 there is an OD $\delta$-complete homogeneity system
 $\vec{\nu}$ of $S$.
The restriction of $(S,\vec{\nu})$ to $\tn{HOD}$
 witnessed that the complement of $B \cap \tn{HOD}$ is
 $\delta$-homogeneously Suslin in $\tn{HOD}$, and hence
 $B \cap \tn{HOD}$ is ${<}\delta$-uB
 in $\tn{HOD}$
 (see Corollary 1.21 of \cite{steel_dm} for the argument
 that $\delta$-homogeneously Suslin implies
 ${<}\delta$-uB).
Since $\Theta$ is a limit of weak partition cardinals,
 $B \cap \tn{HOD}$ is ${<}\Theta$-uB in $\tn{HOD}$.
\end{proof}

\begin{lemma}
Assume the axiom ``V = \tn{Ultimate} L.''
Let $\alpha$ be an ordinal such that $V_\alpha \models \zfc$.
If $x \in \mathbb{R}$,
 then there exists a $\beta < \alpha$
 and an $A \subseteq \mathbb{R}$
 that is universally Baire in $V_\alpha$
 such that
 $$V_\alpha \models x \mbox{ is } \tn{OD}
 \mbox{ in } L_\beta(A,\mathbb{R}).$$
\end{lemma}
\begin{proof}
Note that because we have a proper class
 of Woodin cardinals,
 for any ordinal $\gamma$ such that
 $V_\gamma \models \zfc$,
 we have
 $V_\gamma \models (\forall y)\, y^\#$ exists.

Assume, towards a contradiction,
 that the theorem is false.
Let $\varphi$ be the conjunction of the following:
\begin{itemize}
\item $\zfc$
\item $(\forall y)\, y^\#$ exists
\item $
 (\exists x \in \mbb{R})
 (\forall \beta \in \mbox{Ord})
 (\forall B \in \Gamma^\infty)\,
 x \mbox{ is \textit{not} } \tn{OD} \mbox{ in }
 L_\beta(B,\mathbb{R})$.
\end{itemize}
So $V_\alpha \models \varphi$.
Since
 $$\psi :\Leftrightarrow (\exists \gamma \in \mbox{Ord})\,
 V_\gamma \models \varphi$$
 is a (true) $\Sigma_2$ statement,
 by applying the reflection part
 of the axiom ``$V$ = Ultimate $L$'',
 fix a universally Baire $A \subseteq \mathbb{R}$
 such that 
 $$\mc{H} := \tn{HOD}^{L(A,\mathbb{R})}$$
 satisfies $\psi$.
Fix some $\alpha' \in \mbox{Ord}$
 such that
 $$M := (V_{\alpha'})^\mc{H}$$
 satisfies $\varphi$.
Fix some $x \in \mbb{R}^\mc{H}$ that witnesses
 that $\varphi$ is true in $M$.

Note that
\begin{equation}\label{line_a_od_in_lar}
 \mbox{$x$ is $\tn{OD}$ in $L(A,\mathbb{R})$}
\end{equation}
 (that is, $x \in \mc{H}$).
So therefore we may use condensation
 to fix the smallest ordinal
 $\gamma < \Theta^{L(A,\mathbb{R})}$
 such that
 $x$ is $\tn{OD}$ in $L_\gamma(A,\mathbb{R})$.
The structure $L_\gamma(A,\mathbb{R})$
 can be coded in $L(A,\mathbb{R})$
 by a set of reals.
Using the $\Sigma^2_1$ basis theorem,
 there is a $\Delta^2_1(x)$
 set of reals $E \subseteq \mathbb{R}$
 that codes a structure
 of the form $L_{\tilde{\beta}}(\tilde{B},\mathbb{R})$
 for some ordinal $\tilde{\beta}$ and $\tilde{B} \subseteq \mathbb{R}$
 such that $x$ is $\tn{OD}$ in $L_{\tilde{\beta}}(\tilde{B},\mathbb{R})$.
Also in this invocation of the $\Sigma^2_1$ basis theorem,
 get $\Delta^2_1(x)$ scales for sets projective in $E$.
Using Lemma~\ref{od_scs_is_less_than_theta_ub_in_hod},
 $$\bar{E} := E \cap \mc{H}$$
 is ${<}\Theta^{L(A,\mathbb{R})}$-universally
 Baire in $\mc{H}$.
We also have
\begin{equation}\label{line_elem_sub_1}
 \langle H(\omega_1), \bar{E}, \in
 \rangle^\mc{H}
 \preceq
 \langle H(\omega_1), E, \in \rangle
\end{equation}
 because there are $\Delta^2_1(x)$
 scales for sets projective in $E$, and hence the trees
 of these scales are in \(\mathcal H\)
 (see the proof of
 Proposition \ref{prp:pd_correctness}).
 
Let's say that $\bar{E}$
 codes the following structure in
 $\mc{H}$:
 $$L_\beta(B,\mathbb{R})$$
 for some fixed ordinal $\beta$ and some
 fixed set of reals
 $B$ in $\mc{H}$.
Thus since $x$ is $\tn{OD}$ in $L_{\tilde{\beta}}(\tilde{B},\mathbb{R})$ and
 (\ref{line_elem_sub_1}),
 $$x \mbox{ is } \tn{OD} \mbox{ in }
 L_\beta(B,\mathbb{R}).$$

Because $L_\beta(B,\mathbb{R})$
 is coded by a set of reals in
 $\mc{H}$,
 it follows that
 $$\beta < \Theta^\mc{H}.$$
Note that $\Theta^\mc{H} \not= \Theta^{L(A,\mathbb{R})}$.
Since $\mc{H}$ satisfies $\zfc$,
 $\Theta^\mc{H} = ((2^{\aleph_0})^+)^\mc{H}$
 (and it is a countable ordinal in $V$).
Since $(V_{\alpha'})^\mc{H}$ is a model
 of $\zfc$, we have
 $$\Theta^\mc{H} = ((2^{\aleph_0})^+)^\mc{H} < \alpha'.$$
So now since $\beta < \Theta^\mc{H} < \alpha'$ we have
 $$\beta < \alpha'.$$

Since $\bar{E}$ is
 ${<}\Theta^{L(A,\mathbb{R})}$-universally
 Baire in $\mc{H}$,
 so is $B$.

Next, since $(V_{\alpha'})^\mc{H} \models \varphi$,
 in particular
 $(V_{\alpha'})^\mc{H} \models
 (\forall y)\, y^\#$ exists,
 it must be that
 $$\alpha' \le \Theta^{L(A,\mathbb{R})}.$$
Let us explain.
Because we have a proper class of Woodin cardinals,
 we have that $\tn{HOD}^{L(A,\mbb{R})} = L[X]$
 for some set $X \subseteq \Theta^{L(A,\mbb{R})}$,
 and any model of the form $L[X]$ satisfies that
 $X^\#$ does not exist
 (see Lemma 7.15 of \cite{woodin_in_search_of_ult_l}).

Combining this with the fact that
 $B$ is
 ${<}\Theta^{L(A,\mathbb{R})}$-universally
 Baire in $\mc{H}$,
 we have that $B$ is universally Baire
 in $(V_{\alpha'})^\mc{H}$.

Putting this all together,
 we have that
 $(V_{\alpha'})^\mc{H} \models \varphi$
 as witnessed by $x$,
 but also
 $$(V_{\alpha'})^\mc{H} \models
 (\exists \beta' \in \mbox{Ord})
 (\exists B' \in \tn{uB})\,
 x \mbox{ is } \tn{OD} \mbox{ in }
 L_{\beta'}(B',\mathbb{R}).$$
This is a contradiction.
\end{proof}

Our goal is to show that the axiom ``V = Ultimate-L''
 implies that for every $x \in \mbb{R}$, we have $\zfc +
 $``there exists a Woodin cardinal'' $\proves_\Omega
 (x \in \tn{HOD})$.
We will do this in two halves.
The first half is to show that the
 axiom ``V = Ultimate-L''
 implies that every real number is $\tn{OD}$ in
 $L(A,\mbb{R})$ for some universally Baire $A \subseteq \mbb{R}$.
We will finish that half in the next lemma.
Then, over the next few lemmas we will do the second half
 by showing that whenever $x$ is a real number that
 is $\tn{OD}$ in some model of the form $L(A,\mbb{R})$ where
 $A$ is universally Baire,
 then $\zfc +
 $``there exists a Woodin cardinal'' $\proves_\Omega
 (x \in \tn{HOD})$.

\begin{lemma}
\label{ultl_reals_in_hod_lar}
Assume the axiom ``$V$ = \tn{Ultimate} $L$.''
Let $x \in \mathbb{R}$.
Then there exists a universally Baire
 set of reals $A$ such that
 $x$ is $\tn{OD}$ in $L(A,\mathbb{R})$.
\end{lemma}
\begin{proof}
Let $\alpha$ be an ordinal such that
 the following hold:
\begin{itemize}
\item $V_\alpha \models \zfc$
\item The universally Baire sets of reals of $V_\alpha$
 are the same as the universally Baire sets
 of reals of $V$.
\end{itemize}
Using the lemma above
 for $V_\alpha$,
 fix $\beta < \alpha$ and
 $A \subseteq \mathbb{R}$
 that is universally Baire in $V_\alpha$
 such that
 $x$ is $\tn{OD}$ in $L_\beta(A,\mathbb{R})$.
Note that $A$ is universally Baire in $V$.

It suffices to show that the model $L_\beta(A,\mbb{R})$
 is \tn{OD} in $L(A,\mbb{R})$.
 (We cannot simply use $A$ as a parameter to define it.)
What saves us is that since there is a proper class of Woodin
 cardinals and $A$ is universally Baire (in $V$),
 then $L(A,\mathbb{R}) \models \ad$.
This implies that
 $L_\beta(A,\mbb{R})$ is definable
 in $L(A,\mbb{R})$ using $\beta$ and
 the Wadge rank of $A$ as ordinal parameters.
\end{proof}

Recall the notion of a transitive model of $\zfc$
 being strongly $Z$-closed, as was defined in the introduction
 when we introduced $\Omega$ logic.
An excellent reference for this is \cite{omega_logic_primer}.

\begin{lemma}
\label{proj_correctness_omega_logic}
Assume there is a proper class of
 Woodin cardinals.
Let $B$ be a $\tn{uB}$ set of reals.
Then there is a $\tn{uB}$ set of reals $Z$ such that
 whenever $M$ is a countable transitive model
 of $\zfc$ that is strongly $Z$-closed,
 then $$\langle H(\omega_1), \in B \cap M \rangle^M\
 \preceq \langle H(\omega_1), \in, B \rangle.$$
\end{lemma}
\begin{proof}
Let $\varphi$ be a formula with two real
 number parameters
 about the structure
 $\langle H(\omega_1), \in, B \rangle$.
Let $A \subseteq \baire \times \baire$
 be $$A = \{ (x,y) : \varphi(x,y) \}.$$
Let $M$ be a countable transitive model of $\zfc$.
Suppose that $M$ contains a tree for $A$.
That is, $T_\varphi \in M$ for some $T_\varphi$ where
 $A = p[T_\varphi]$.
By using leftmost branches,
 for any $x \in \mathbb{R} \cap M$,
 if there exists some $y \in \mathbb{R}$
 such that $\varphi(x,y)$,
 then there will be such a $y$ in $M$.

So now if $M$ contains such a tree for each
 formula $\varphi$ with two free variables,
 then by the Tarski criterion we will have
 what we want.
Finally, we can get $M$ to contain such trees
 because uB sets of reals have uB scales
 (because we are assuming a proper class
 of Woodin cardinals).

Let us explain this last part.
Suppose $A \subseteq \baire \times \baire$
 is projective in $B$.
Since there is a proper class of Woodin cardinals
 and $B$ is $\tn{uB}$, it follows that
 $A$ is $\tn{uB}$ as well.
Also since there is a proper class of Woodin cardinals,
 fix a scale $\vec{\eta}$ for $A$ that is $\tn{uB}$.
There is a $\tn{uB}$ set of reals $C_A$
 with two properties:
 1) any countable transitive model $M$ of $\zfc$
 that is strongly $C_A$-closed is strongly $A$-closed, and
 2) whenever $M$ is a countable transitive model of $\zfc$
 that is strongly $C_A$-closed,
 then $C_A \cap M \in M$ can be used in $M$
 to define a scale for $A \cap M \in M$.
Basically, we are intersecting scales with
 countable transitive models of $\zfc$
 to get scales in those models.

So now for every set $A \subseteq \baire \times \baire$
 projective in $B$,
 we have a certain $\tn{uB}$ set of reals $C_A$.
We can put all these (countably many) $C_A$ sets together
 to get a single $\tn{uB}$ set of reals $Z$.
That is, we get a $\tn{uB}$ set of reals $Z$ such that whenever
 $M$ is a countable transitive model of $\zfc$
 such that $M$ is strongly $Z$-closed,
 then $M$ is strongly $C_A$-closed for
 every set $A \subseteq \baire \times \baire$ projective in $B$.
This completes the proof.
\end{proof}

\begin{lemma}
\label{ub_in_ctm_omega_logic_primer}
    Suppose \(A\) is a \tn{uB} set of reals and
    \(\vec\varphi\) and
    \(\vec\psi\)
    are \tn{uB} scales on \(A\) and its complement.
    Let \(R_{\vec\varphi}\) and \(R_{\vec \psi}\) be the corresponding relations on
    \(\omega\times \mathbb R\times \mathbb R\).
    If \(M\) is a
    strongly \(R_{\vec\varphi}\times R_{\vec \psi}\)-closed countable transitive model of \textnormal{ZFC},
    then \(A \cap M\) is \tn{uB} in \(M\).
    \begin{proof}
        See the proof of \cite[Theorem 2.37]{omega_logic_primer}.
    \end{proof}
\end{lemma}

Let us go on a brief tangent and give a result
 which shows how $\Omega$-reals fit into the
 bigger picture:
\begin{proposition}
\label{reals_in_mice_are_omega_reals}
Let $\mc{M}$ be a countable premouse that has
 a universally Baire $\omega_1$-iteration strategy
 $\Sigma \subseteq \mbb{R}$.
Let $x \in \mbb{R} \cap \mc{M}.$
Then $\tn{ZFC} \proves_\Omega (x \in \tn{HOD})$.
\end{proposition}
\begin{proof}
Using Lemma~\ref{ub_in_ctm_omega_logic_primer},
 fix a universally Baire set $\Sigma' \subseteq \mbb{R}$
 such that any countable transitive model $N$ of $\zfc$
 which is $\Sigma'$-closed is also
 $\Sigma$-closed and $\Sigma \cap N$
 is universally Baire in $N$.
We claim that $\Sigma'$ is an $\Omega$ proof
 that $\zfc \proves_\Omega (x \in \tn{HOD})$.

To see why, let $N$ be any countable transitive
 model of ZFC which is $\Sigma'$-closed.
Since $N$ is also $\Sigma$-closed,
 by definition $\Sigma \cap N \in N$.
This implies that $\Sigma \cap N$
 is an $\omega_1$-iteration strategy for $\mc{M}$,
 from the point of view of $N$.
But also $\Sigma \cap N$ is universally Baire in $N$.
So $N$ satisfies that $\mc{M}$ has
 a universally Baire $\omega_1$-iteration strategy
 and this implies that $\mc{M}$ is
 $(\omega_1 + 1)$-iterable in $N$
 (see Lemma 1.21 of \cite{sargsyan_dimt}).
Then recall that any real in a countable premouse
 that is $(\omega_1+1)$-iterable
 is $\Delta^2_2$
 in a countable ordinal
 (see Corollary 1.12 of \cite{sargsyan_dimt})
 and hence is $\tn{OD}$.
Thus $x$ is $\tn{OD}$ in $N$.
\end{proof}

Here is a generally useful fact
 which we will use soon.
Given a cardinal $\delta$,
 let $\Gamma^\delta$ represent the collection of
 ${{\le} \delta}$-universally Baire sets of reals. 
\begin{lemma}
\label{one_woodin_is_enough_for_wadge_det}
Let $\delta$ be a Woodin cardinal.
Then $\Gamma^\infty$ is a Wadge initial segment
 of $\Gamma^\delta$.
\end{lemma}
\begin{proof}
Recall that $\Gamma^\infty$ is closed under
 continuous preimages.
It suffices to show that $\Gamma^\delta$ has the usual
 Wadge reducibility structure that the $\mc{P}(\baire)$
 sets of models of $\ad$ have.
A crucial fact we need is that every
 ${{\le} \delta}$-universally Baire set of reals is determined
 (see Theorem 6.17 of \cite{neeman_handbook_chapter}).
Given sets $A, B \in \Gamma^\delta$,
 we can consider the Lipschitz game $G(A,B)$
 that was introduced by Wadge
 (see \cite{wadge_degrees_intro}).
If $G(A,B)$ is determined, then either
 $A$ is Wadge reducible to $B$ or
 the complement of $B$ is Wadge reducible to $A$.
Also, one can check that since $A, B \in \Gamma^\delta$,
 the payoff set of the game $G(A,B)$ is also in $\Gamma^\delta$,
 and so it is determined.
This establishes what we want.
\end{proof}

\begin{lemma}
\label{od_in_lar_implies_omega_real}
Assume there is a proper class
 of Woodin cardinals.
Let $x \in \mathbb{R}$.
Suppose that $A$ is a universally Baire
 set of reals such that
 $x$ is \tn{OD} in $L(A,\mathbb{R})$.
Then $\tn{ZFC}$ $+$ ``there exists a Woodin cardinal''
 $\proves_\Omega (x \in \tn{HOD})$.
\end{lemma}
\begin{proof}
Let $B := (A,\mathbb{R})^\#$.
Notice that the statement
 $$x \in \tn{HOD}^{L(A,\mathbb{R})}$$
 is encoded into $B$.
Using Lemma~\ref{proj_correctness_omega_logic},
let $Z_1$ be a uB set
 of reals such that whenever $M$ is strongly $Z_1$-closed,
 then $$\langle H(\omega_1), \in B \cap M \rangle^M\
 \preceq \langle H(\omega_1), \in, B \rangle,$$
 and so
 $$M \models x \in \tn{HOD}^{L(A \cap M,\mathbb{R})}.$$
Here we are using $B \cap M$
 to recover $A \cap M$.

Next let $\vec{\varphi}$ and $\vec{\psi}$ be uB
 scales on A and its complement.
Let $Z_2 := Z_1 \times R_{\vec{\varphi}}
 \times R_{\vec{\psi}}$.
Certainly $Z_2$ is uB.

Let $M$ be a countable transitive
 model of the theory 
 $\zfc$ + ``there exists a Woodin cardinal''
 such that $M$ is strongly $Z_2$-closed.
Let $\bar{A} := A \cap M$.
Since $M$ is strongly $Z_1$-closed,
 we have
 $$M \models x \in \tn{HOD}^{L(\bar{A},\mathbb{R})}.$$
Since $M$ is strongly
 $R_{\vec{\varphi}}
 \times R_{\vec{\psi}}$-closed,
 so by Lemma~\ref{ub_in_ctm_omega_logic_primer}
 we get that $\bar{A}$ is uB in M.

Now since $M$ has a Woodin cardinal $\delta$,
 the model $L(\bar{A},\mathbb{R})$ of $M$
 is definable in $M$ using an ordinal as a parameter.
See Lemma~\ref{one_woodin_is_enough_for_wadge_det}.
That is, $M$ can define its hierarchy
 of ${\le}\delta$-uB sets of reals,
 and the $L(\bar{A},\mathbb{R})$ model
 is equal to any $L(C,\mathbb{R})$ model
 where $C$ is any ${\le}\delta$-uB set of reals
 of the same Wadge rank as $\bar{A}$.

Since the model $L(\bar{A},\mathbb{R})$ of $M$
 is OD in $M$, it follows that $x$ is OD in $M$.
This is what we wanted to show.
\end{proof}

Here is an alternate way to prove
 something weaker, but we think the
 proof is worth including:

\begin{lemma}
\label{in_lar_to_omega_valid}
Assume there is a proper class
 of Woodin cardinals.
Let $x$ be a real.
Suppose $A$ is a universally Baire
 set of reals such that
 $x$ is $\tn{OD}$ in $L(A,\mathbb{R})$.
Then $\zfc$ $+$ ``there exists a Woodin
 cardinal'' $\models_\Omega (x \in \tn{HOD})$.
\end{lemma}
\begin{proof}
Let $V[G]$ be a forcing extension of $V$.
Let $A^*$ be the reinterpretation
 of $A$ in $V[G]$.
Since there is a proper class of Woodin
 cardinals, fix an elementary embedding
 $$j : L(A,\mathbb{R}) \to
 L(A^*,\mathbb{R})^{V[G]}.$$
Note that $j(x) = x$,
 so since $j$ is elementary we get that
 $$x \mbox{ is } \tn{OD} \mbox{ in }
 L(A^*,\mathbb{R})^{V[G]}.$$
By condensation, fix
 the smallest ordinal
 $\beta < (\Theta^{L(A^*,\mathbb{R})})^{V[G]}$
 such that
 $$x \mbox{ is } \tn{OD} \mbox{ in }
 L_\beta(A^*,\mathbb{R})^{V[G]}.$$
Now fix an ordinal $\alpha$
 such that $$(V_\alpha)^{V[G]}
 \models \mbox{$\zfc$ + ``there exists a Woodin cardinal.''}$$
We must show that
 $(V_\alpha)^{V[G]} \models
 (x \in \tn{HOD})$.
Because $V_\alpha^{V[G]} \models \zfc$,
 we have the $\Theta$ of $L(A^*,\mathbb{R})^{V[G]}$
 is $< \alpha$,
 and hence
 $$\beta < \alpha.$$

Since $A^*$ is universally Baire in $V[G]$,
 it follows that $A^*$ is universally Baire
 in $(V_\alpha)^{V[G]}$.
However there might be sets of reals
 that are universally Bare in $(V_\alpha)^{V[G]}$
 that are \textit{not}
 universally Baire in $V[G]$.
It is also not clear how to define
 the collection of universally Baire
 sets of reals of $V[G]$
 from inside $(V_\alpha)^{V[G]}$.

This is where the Woodin cardinal of
 $(V_\alpha)^{V[G]}$
 comes in.
Let $\delta < \alpha$
 be the least Woodin cardinal of
 $(V_\alpha)^{V[G]}$
 and therefore of $V[G]$ as well.
Note that $\delta$ is definable in
 $(V_\alpha)^{V[G]}$.
Recall that $(\Gamma^\infty)^{V[G]}$ is the collection
 of universally Baire sets of reals of
 $V[G]$ and recall that $(\Gamma^\delta)^{V[G]}$
 be the collection of
 ${\le} \delta$-universally Baire sets of reals
 of $V[G]$.
Note that $(\Gamma^\delta)^{V[G]}$
 is definable in $(V_\alpha)^{V[G]}$.

Because $\delta$ is Woodin,
 we get that $(\Gamma^\infty)^{V[G]}$ is a Wadge initial
 segment of $(\Gamma^\delta)^{V[G]}$
 (see Lemma~\ref{one_woodin_is_enough_for_wadge_det}).
Thus the model $L_\beta(A^*,\mathbb{R})^{V[G]}$
 is definable in $(V_\alpha)^{V[G]}$
 using only $\beta$ and the Wadge rank
 $\gamma$ of $A^*$ as parameters.
More precisely,
 $(V_\alpha)^{V[G]}$ satisfies that
 $L_\beta(A^*,\mathbb{R})$
 is the unique model of the form
 $L_\beta(B,\mathbb{R})$
 where $B$ is any set in $(\Gamma^\delta)^{V[G]}$
 of Wadge rank $\gamma$.

Finally, since
 $x$ is $\tn{OD}$
 in $L_\beta(A^*,\mathbb{R})^{V[G]}$
 and that model is $\tn{OD}$ in
 $(V_\alpha)^{V[G]}$,
 it follows that
 $(V_\alpha)^{V[G]} \models
 (x \in \tn{HOD})$.
\end{proof}

The axiom ``$V$ = \tn{Ultimate} $L$''
 thus implies that every real number
 is an $\Omega$-real:

\begin{theorem}[Woodin]
\label{ult_l_implies_exery_real_is_omega}
Assume the axiom ``$V$ = \tn{Ultimate} $L$.''
Let $x \in \mathbb{R}$.
Then $\zfc$ $+$ ``there exists a Woodin
 cardinal'' $\proves_\Omega (x \in \tn{HOD})$.
\end{theorem}
\begin{proof}
By Lemma~\ref{ultl_reals_in_hod_lar},
 fix a universally Baire set of reals
 $A \subseteq \mathbb{R}$ such that
 $x$ is $\tn{OD}$ in $L(A,\mathbb{R})$.
By Lemma~\ref{od_in_lar_implies_omega_real}
 this implies that $x$ is an $\Omega$-real.
\end{proof}

Here is how this section relates
 to our main investigation:

\begin{proposition}
\label{phi2_implies_super_abs_fail}
Assume there is a Woodin cardinal $\delta$
 and a proper class of inaccessible cardinals.
Assume that
 for each $x \in \mathbb{R}$,
 $$\tn{ZFC} + \mbox{``there exists a Woodin cardinal''}
 \proves_\Omega (x \in \tn{HOD}).$$
Then generic
 $(\forall^\baire)(\Sigma^2_1)^{\tn{uB}}$
 absoluteness fails.
\end{proposition}
\begin{proof}
Let $G$ be Cohen generic over $V$.
By Lemma~\ref{cohen_forcing_breaks_phi2},
 the forcing extension does not satisfy
 the $(\forall^\baire)(\Sigma^2_1)^{\tn{uB}}$
 statement that
 for every real number $x$,
 $$\tn{ZFC} + \mbox{``there exists a Woodin cardinal''}
 \proves_\Omega (x \in \tn{HOD}).$$
Indeed, this fails for any real in the extension
 that is not in the ground model.
\end{proof}

Hence by the theorem above,
 the axiom ``$V$ = Ultimate $L$'' implies
 that generic $(\forall^\baire)(\Sigma^2_1)^{\tn{uB}}$
 absoluteness fails.
Thus, assuming the Ultimate $L$ \textit{conjecture},
 no conventional large cardinal hypothesis
 can imply generic
 $(\forall^\baire)(\Sigma^2_1)^{\tn{uB}}$
 absoluteness.

\section{Acknowledgements}

We would like to thank Paul Larson,
 Grigor Sargsyan, and John Steel
 for discussions relating to this paper.
We would also especially like to thank
 the referee who was very thorough
 and made many helpful suggestions.


\begin{thebibliography}{99}


\bibitem{wadge_degrees_intro}
A. Andretta, A. Louveau.
Wadge degrees and pointclasses. Introduction to Part III.
In: Kechris AS, Löwe B, Steel JR, (Eds.)
Wadge Degrees and Projective Ordinals:
The Cabal Seminar, Volume II.
Lecture Notes in Logic. Cambridge University Press,
Cambridge, England.
(2011), 3-23.

\bibitem{omega_logic_primer}
J. Bagaria, N. Castells, and P. Larson.
An Omega-logic primer.
Set Theory, CRM 2003-2004,
Birkhauser (2006), 1-28.

\bibitem{davis_hod_dichotomy}
 J. Davis, W. H. Woodin, D. Rodríguez.
 The HOD Dichotomy.
 In: Cummings J, Schimmerling E, eds. Appalachian Set Theory:
 2006–2012. London Mathematical Society Lecture Note Series.
 Cambridge University Press. (2012) 397-419.

\bibitem{feng_ub}
Q. Feng, M. Magidor, and W. H. Woodin.
Universally Baire Sets of Reals.
In Judah, H., Just, W., Woodin, H. (Eds.) Set Theory of the Continuum. Mathematical Sciences Research Institute Publications,
vol 26. Springer, New York, NY. (1992). 203-242.

\bibitem{hamkins_cover_approx}
J. D. Hamkins.
Extensions with the approximation and cover properties have no new large cardinals.
Fundamenta Mathematicae, 180 (2003) No.3, 257-277.

\bibitem{jackson_intro}
S. Jackson.
Projective ordinals. Introduction to Part IV.
In: Kechris AS, Löwe B, Steel JR, (Eds.)
Wadge Degrees and Projective Ordinals:
The Cabal Seminar, Volume II.
Lecture Notes in Logic. Cambridge University Press,
Cambridge, England.
(2011), 199-269.

\bibitem{jech_book}
T. Jech.
Set Theory, The Third Millennium Edition,
Revised and Expanded.
Springer, New York, NY, (2002).

\bibitem{kanamori_book}
A. Kanamori.
The Higher Infinite: Large Cardinals in Set Theory from Their Beginnings.
Springer, Berlin, (2009).

\bibitem{kechris_ad}
A. Kechris A, E. Kleinberg E, Y. Moschovakis, W.H. Woodin.
The axiom of determinacy, strong partition properties, and nonsingular measures. In: Kechris AS, Löwe B, Steel JR, eds. Games, Scales and Suslin Cardinals: The Cabal Seminar, Volume I. Lecture Notes in Logic. Cambridge University Press; 2008:333-354.

\bibitem{larson_tower_book}
P. Larson.
The Stationary Tower: Notes on a Course by W. Hugh Woodin.
American Mathematical Society, Providence RI, (2004).

\bibitem{larson_ad_book}
P. Larson.
Extensions of the Axiom of Determinacy.
American Mathematical Society, Providence RI, (2023).

\bibitem{proof_of_proj_det}
D. Martin and J. Steel.
A Proof of Projective Determinacy.
Journal of the American Mathematical Society,
2 (1989), No.1, 71–125.

\bibitem{tree_of_m_scale_is_homog}
D. Martin and J. Steel.
The tree of a Moschovakis scale is homogeneous.
In: Kechris AS, L\"{o}we B, Steel J, (Eds.)
Games, Scales and Suslin Cardinals:
The Cabal Seminar, Volume I.
Lecture Notes in Logic. Cambridge University Press,
Cambridge, England.
(2008), 404-420.

\bibitem{martin_solovay_basis_theorem}
D. Martin and R. Solovay.
A basis theorem for $\Sigma^1_3$ sets of reals,
Ann. Math. 89 (1969), 138-159.

\bibitem{moschovakis_book}
Y. Moschovakis.
Descriptive set theory, second edition.
American Mathematical Society, United States, (2009).

\bibitem{a_coding_lemma_paper}
Y. Moschovakis.
Determinacy and prewellorderings of the continuum.
In Studies in Logic and the Foundations of Mathematics,
Bar-Hillel Y. (Ed.)
Elsevier, 59 (1970), 24-62.

\bibitem{neeman_handbook_chapter}
I. Neeman.
Determinacy in $L(\mathbb{R})$.
In Handbook of Set Theory.
In Foreman, M. and Kanamori A., (Eds.),
Springer, New York, NY, (2010), 1887–1950.

\bibitem{sargsyan_msc_book}
G. Sargsyan.
Hod Mice and the Mouse Set Conjecture.
American Mathematical Society, Providence RI, (2015).

\bibitem{sargsyan_dimt}
G. Sargsyan.
Descriptive inner model theory.
Bulletin of Symbolic Logic, 19 (2013), No.1, 1-55.

\bibitem{steel_dm}
J. Steel.
The derived model theorem.
In: Cooper SB, Geuvers H, Pillay A, Väänänen J, (Eds.)
Logic Colloquium 2006.
Lecture Notes in Logic. Cambridge University Press,
Cambridge, England. 
(2009), 280-327.

\bibitem{steel_intro_iterated_ult}
J. Steel.
An introduction to iterated ultrapowers.
In Chong, C., Feng, Q., Slaman, A., Woodin, W. H., Yang, Y., (Eds.)
Forcing, Iterated Ultrapowers, and Turing Degrees.
 Lecture Notes Series, Volume 29,
Institute for Mathematical Sciences, National University of Singapore,
 (2015).

\bibitem{steel_cabal_3}
J. Steel.
Ordinal definability in models of determinacy. Introduction to Part V.
In Kechris, A, L\"{o}we, B, and Steel, J (Eds.)
Ordinal Definability and Recursion Theory: The Cabal Seminar, Volume III.
Lecture Notes in Logic. Cambridge University Press,
Cambridge, England.
(2016), 3-48.

\bibitem{wilson_ub_and_gen_abs}
T. Wilson.
Universally Baire sets and generic absoluteness.
Journal of Symbolic Logic 82 (2017) No.4, 1229–1251.

\bibitem{woodin_ch2}
W. H. Woodin.
The Continuum Hypothesis, Part II.
Notices of the AMS
48 (2001), No.7, 681-690.

\bibitem{woodin_in_search_of_ult_l}
W. H. Woodin.
In search of Ultimate-L,
 the 19th Midrasha Mathematicae Lectures.
Bulletin of Symbolic Logic, 23 (2017), No.1, 1-109.

\bibitem{woodin_sem}
W. H. Woodin.
Suitable Extender Models I.
Journal of Mathematical Logic 10 (2010), No.1, 101-339.


\end{thebibliography}
\end{document}